\documentclass[a4paper,11pt]{amsart}  
\usepackage{amstext}
\usepackage{amsmath}
\usepackage{ifthen}
\usepackage{amscd}
\usepackage{amssymb}
\usepackage[english]{babel}
\usepackage[T1]{fontenc}
\usepackage[utf8]{inputenc}
\usepackage[all]{xy}
\usepackage{graphicx}
\usepackage{enumerate}
\usepackage{xspace}
\usepackage{epic}
\usepackage{xcolor}

\usepackage[
	hypertexnames=false,
	hyperindex,
	pagebackref,  
	pdftex,
	bookmarks=false,
	colorlinks,
	linkcolor=blue,
	citecolor=red,
	urlcolor=red,
]{hyperref}
\newtheorem{prop}{Proposition}[section]
\newtheorem{coro}[prop]{Corollary}

\newtheorem{lem}[prop]{Lemma}
\newtheorem{rem}[prop]{Remark}
\newtheorem{exe}[prop]{Example}
\newtheorem{defi}[prop]{Definition}
\newtheorem{theo}[prop]{Theorem}

\theoremstyle{plain}  
\newtheorem{theorem}{Theorem}[section]      
\newtheorem{lemma}[prop]{Lemma}     
      
\newtheorem{proposition}[prop]{Proposition}

\newtheorem{definition}[prop]{Definition}         
\theoremstyle{remark}

\newcommand{\Q}{\mathbb Q}
\newcommand{\C}{\mathbb C}
\newcommand{\N}{\mathbb N}

\newcommand{\Z}{\mathbb Z}

\newcommand{\Mg}{\mathcal{M}_g}

\newcommand{\Mgan}{\mathcal{M}_{g}^{an}}

\newcommand{\Mgone}{\mathcal{M}_{g}^1}
\newcommand{\Mgonean}{\mathcal{M}_{g}^{1 \ an}}

\newcommand{\Mgbar}{\overline{\mathcal{M}_g}}
\newcommand{\Mgbaran}{\overline{\mathcal{M}^{an}_g}}

\newcommand{\Mgobar}{\overline{\mathcal{M}^1_{g}}}
\newcommand{\Mgobaran}{\overline{\mathcal{M}^{1 \ an}_{g}}}

\newcommand{\Aut}{{\rm{Aut}}}  
\newcommand{\Out}{{\rm{Out}}}  
\newcommand{\Mod}{{\rm{Mod}}} 
\newcommand{\PMod}{{\rm{PMod}}}           

\newcommand{\ro}{{\widetilde{\rho}}}

\newenvironment{prv}{Proof:}{$\Box$}

\begin{document}

\title[Examples of Orbifold K\"ahler Groups]{Orbifold K\"ahler Groups related to Mapping Class groups}

\date{December, 13, 2021}

\author[P.~Eyssidieux]{Philippe Eyssidieux}
\address{Institut Fourier,
Laboratoire de Math\'ematiques UMR 5582,   
Universit\'e Grenoble Alpes,
CS 40700, 
38058 Grenoble, France}
\email{philippe.eyssidieux@univ-grenoble-alpes.fr}
\urladdr{http://www-fourier.univ-grenoble-alpes.fr/$\sim$eyssi/}
\thanks{The  research of P. E.  was partially supported by the ANR project Hodgefun ANR-16-CE40-0011-01.}

\author[L.~Funar]{Louis Funar}
\address{Institut Fourier,
Laboratoire de Math\'ematiques UMR 5582,   
Universit\'e Grenoble Alpes,
CS 40700, 
38058 Grenoble, France}
\email{louis.funar@univ-grenoble-alpes.fr} 
\urladdr{http://www-fourier.univ-grenoble-alpes.fr/$\sim$funar/}

\begin{abstract} We construct certain orbifold compactifications of the 
moduli stack of pointed stable curves over $\C$ and study their fundamental groups 
by means of their quantum representations. This enables to construct interesting K\"ahler groups  and 
to settle most of the candidates for a counter-example to the Shafarevich conjecture on holomorphic convexity proposed in 1998 by Bogomolov and Katzarkov, using 
TQFT representations of the mapping class groups. 

\vspace{0.2cm}
 
\vspace{0.2cm} 
\end{abstract}

\maketitle

\tableofcontents

\section{Introduction and statements}

We denote by $\Mod(\Sigma_g)$ the mapping class group of a closed oriented surface $\Sigma_g$ of genus $g$. Further 
$\Mod(\Sigma_g^1)$ denotes  the mapping class group of the pair $(\Sigma_g,\ast)$, namely the group of isotopy classes of  orientation preserving homeomorphisms
of $\Sigma_g$ fixing $\ast$. It is well-known that $\Mod(\Sigma_g^1)$ is isomorphic to the mapping class group of the punctured surface $(\Sigma_g,\ast)$. 
 By forgetting the marked point  one obtains a 
surjective homomorphism $\Mod(\Sigma_g^1)\to \Mod(\Sigma_g)$ which is part of the Birman exact sequence: 
\[ 1\to \pi_1(\Sigma_g) \to \Mod(\Sigma_g^1)\to \Mod(\Sigma_g)\to 1\]
 
The Dehn-Nielsen-Baer theorem states that 
the map associating to $\varphi\in \Mod(\Sigma_g^1)$ the automorphism $\varphi_*:\pi_1(\Sigma_g,\ast)\to \pi_1(\Sigma_g,\ast)$
provides an  isomorphism between $\Mod(\Sigma_g^1)$ and $\Aut^+(\pi_g)$ and induces an isomorphism 
$\Mod(\Sigma_g)\to \Out^+(\pi_g)$.
Furthermore, the diagram below is commutative: 
\[\begin{array}{ccccccccc}
 1& \to & \pi_1(\Sigma_g) & \to & \Mod(\Sigma_g^1)&\to & \Mod(\Sigma_g) &\to& 1\\
  &  & \downarrow & & \downarrow & & \downarrow & & \\
 1 & \to & \pi_1(\Sigma_g) & \to & \Aut^+(\pi_g) & \to &\Out^+  (\pi_g) & \to & 1 \\
 \end{array}
 \] 
where the top horizontal line is the Birman exact sequence.

The groups $\Mod(\Sigma_g)$ and $\Mod(\Sigma_g^1)$  occur in complex algebraic geometry as the 
fundamental groups of the analytifications of the moduli stacks of curves $\Mgan$ and $\Mgonean$ respectively. 
More generally, let  $\Mod(\Sigma_g^n)$ be the mapping class group of the 
genus $g$ orientable surface $\Sigma_g^n$ with $n$ punctures (or marked points) and 
$\PMod(\Sigma_g^n)$ denote the {\em pure} mapping class group of those classes which fix 
{\em pointwise} the punctures. Then $\PMod(\Sigma_g^n)$ and  $\Mod(\Sigma_g^n)$ also occur as the 
fundamental groups of the analytification $\mathcal{M}^{n \ an}_{g}$ of the moduli stacks of curves 
 and 
$[S_n\backslash{\mathcal{M}^{n \ an}_{g}}]$, respectively, where $S_n$ acts on $\mathcal{M}^{n \ an}_{g}$ permuting the markings.  
These stacks are separated, smooth, Deligne-Mumford and their moduli spaces are quasiprojective, 
but non proper except in the trivial case  $g=0,n=3$. 

Constructing stacky compactifications of such objects 
can be used an intermediate step to construct interesting K\"ahler groups \cite{ajm,EysOrb2} and  
in the present case, to recast some of the basic ideas of Teichm\"uller level structures \cite{ BoPi,Bry, Loo}.

For every integer $p\geq 2$ we introduce the subgroup $\Mod(\Sigma_g^n)[p]\subset \Mod(\Sigma_g)$ 
generated by the $p$-th powers of Dehn twists along simple closed curves in $\Sigma_g^n$. 
Note that  $\Mod(\Sigma_g^n)[p]\subset \PMod(\Sigma_g^n)$ is a normal subgroup of $\Mod(\Sigma_g)$ and in fact a characteristic one. 
More generally, there are only finitely many, say $N_{g,n}$ 
conjugacy classes of Dehn twists, or equivalently, 
distinct orbits of essential simple closed curves on $\Sigma_g^n$ under the mapping class group action.
For instance $N_{g}=\lfloor \frac{g}{2}\rfloor+1$, $N_{g,1}=g$. 
Simple closed curves orbits are determined by the homeomorphism type of the 
complementary subsurface, which might be either connected for non-separating curves 
or disconnected and hence determined by the set/pair of genera of its two components, for separating curves. 
Fix an enumeration of these homeomorphism types starting with the non-separating one. For each  vector $\mathbf k=(k_0,k_1,k_2,\ldots,k_{N_{g,n}-1})$ we define  
$\Mod(\Sigma_g^n)[\mathbf k]$ as the (normal) subgroup generated 
by $k_0$-th powers of Dehn twist along non-separating simple closed curves and 
$k_j$-th powers of Dehn twists along  simple closed curves of type $j$.
As a shortcut we use $\Mod(\Sigma_g^n)[k;m]$ for $\mathbf k=(k,m,m,m\ldots,m)$ and 
 $\Mod(\Sigma_g^n)[k;-]$ for $\mathbf k=(k;)$, where $k_i$ are absent for $i>0$.  

The groups $\Mod(\Sigma_{g})\slash \Mod(\Sigma_{g})[p]$ were introduced and studied in \cite{F} and  $\Mod(\Sigma_{g}^1)\slash \Mod(\Sigma_{g}^1)[p]$ in \cite{LF}. 
In \cite{AF} the authors proved that they are virtually K\"ahler groups by considering 
the Deligne-Mumford compactifications of the moduli spaces of curves for some level structures 
for which they are smooth projective varieties (see e.g. \cite{Boggi}). 
To show that these groups are actually K\"ahler we study the images of the isotropy groups 
under quantum representations of mapping class groups in order to be able to use the uniformizability criterion of \cite{vjm}. 
Note that quantum representations of these groups cannot be injective unless their images are non-arithmetic subgroups of higher rank Lie groups (see \cite{FP2}).   

Our first main result is:

\begin{theorem}\label{kahler}
Let $g,n \in \N$ such that $2-2g-n<0$, $(g,n)\not = (1,1), (2,0)$.

For every ${\mathbf{k}} \in \mathbb{N}_{>0}^{N_{g,n}}$ the quotient $\PMod(\Sigma_{g}^n)\slash \Mod(\Sigma_{g}^n)[\mathbf k]$  
is the fundamental group of a compact K\"ahler orbifold ${\overline{\mathcal{M}_{g}^{n \ an}}}[\mathbf k]$ compactifying $\mathcal{M}_{g}^{n \ an}$
 whose moduli space is the moduli space  
${\overline{{M}_{g}^{n \ an}}}$ of stable $n$-punctured curves of genus $g$.

If $p\geq 5$ is odd, ${\overline{\mathcal{M}_{g}^{n \ an}}}[p]$  is uniformizable, hence 
$\PMod(\Sigma_{g}^n)\slash \Mod(\Sigma_{g}^n)[p]$ is a K\"ahler group, i.e. it occurs as the fundamental group of a compact  K\"ahler manifold.

If $p\ge 10$ is even and $(g,n,p)\neq (2,0,12)$, $\overline{{\mathcal M}_g^{n \; an}}[2p,  p/{\rm g.c.d.}(p,4)]$ is uniformizable and
$\Mod(\Sigma_g^n)/\Mod(\Sigma_g^n)[2p,  p/{\rm g.c.d.}(p,4)]$
is a K\"ahler group.
\end{theorem}

It is an important open question whether one can construct a complex projective surface whose fundamental group is infinite   and
has no complex linear finite dimensional  representation with infinite image. For most $p$, 
one knows that $\Mod(\Sigma_g)\slash \Mod(\Sigma_g)[p]$ is infinite  but the proof uses a linear representation \cite{F}.
This representation can be lifted to a representation of $\PMod(\Sigma_g^n)\slash \Mod(\Sigma_g^n)[p]$
with infinite image. Hence our construction gives infinite K\"ahler groups but will not solve that question. 
The fact that the representation comes from a Reshetikhin-Turaev TQFT is however striking from an algebro-geometric perspective since it is not known
to the authors if these representations are motivic  or rigid.
 
However, our construction was motivated by its relation to certain complex projective surfaces proposed by Bogomolov and Katzarkov \cite{BK} as  potential counterexamples to the Shafarevich
conjecture on holomorphic convexity  that remained unsettled up to now. 

These surfaces are obtained in the following way. Let $\pi:S\to C$ be a smooth projective surface fibered over a curve such that the regular 
fibers have genus $g\ge 2$ and the singular fibers  are reduced and have only nodes. Let $N\in \N^*$ be a positive integer. 
Perform a base change along a covering $C'\to C$  with uniform ramification index $N$ at the images of the  singular fibers. Then, for most $N$, and if $\pi$ is stable,  the surface
$S\times_C C'$,  which has a finite number of singular points of type $A_{N-1}$ lying over the nodes in the singular fibers,  has a
quasi-\'etale
cover $S(N)$ which is a smooth projective surface inheriting from $\pi$ a fibration over some curve. The image $I_N$ of the fundamental group of a general fiber $F$ of this fibration in $\pi_1(S(N))$ is 
what they call a Burnside type quotient of $\pi_g$, and is
an infinite group if $N$ is large enough. 
The only \lq obvious\rq \ linear representations of $\pi_1(S(N))$ with infinite image have finite image when restricted to $I_N$. If $\pi$ is not stable, or $N$ is even,  then $S(N)$ may be only a non developable orbisurface. 
 
However, we shall see that quantum representations give rise to representations of $\pi_1(S(N))$, the reason being that $S(N)$ maps non trivially 
to $\Mgobaran[N]$. The recent result  \cite{KoSa} translates into  the fact 
that most of them have an  infinite
image on its general fiber. A remarkable property of TQFT representations, their functoriality  when passing to subsurfaces,  actually enables to settle the question raised by \cite{BK}
in most cases:

\begin{theorem}\label{bogo-katza}
If $N \not\in \{1, 2,3,4,6,8,12, 16,20,24\}$  
the universal covering space of the Bogomolov-Katzarkov  surface ${S}(N)$ is holomorphically convex. 
Moreover, if $\pi:S\to C$ is a stable curve over $C$, then  it is Stein. 
 \end{theorem}
 
 If $S(N)$ is not developpable, its universal covering space is a complex analytic Deligne-Mumford stack, and we say it is Stein, resp. holomorphically convex, if such is its moduli space (a normal complex analytic space with only quotient singularities). 
 
We believe  the examples constructed here to be extremely interesting from the point of view of  uniformization in several complex variables. 
It does not seem to be trivial task to settle the Shafarevich and the Toledo conjectures for ${\overline{\mathcal{M}_{g}}}[p]$. 
We also conjecture Kazhdan's property (T) for its fundamental group since it would follow from Kazhdan's property (T) for the mapping class group - in particular we would have $q^{+}({\overline{\mathcal{M}_{g}}}[p])=0$, which follows from 
the Ivanov conjecture. 
 We hope to come back on these questions in future work. 
 
 \vskip 0.3cm

In the next section we define the Mumford-Deligne stack  ${\overline{\mathcal M_g^{n \; an}}}[\mathbf{k}]$ 
and show that it is an orbifold with the fundamental group  $\PMod(\Sigma_{g}^n)\slash \Mod(\Sigma_{g}^n)[\mathbf{k}]$. In section 3, we introduce quantum representations of mapping class 
groups, compute 
the orders of Dehn twists and prove they factor through $\Mod(\Sigma_{g}^n)\slash \Mod(\Sigma_{g}^n)[p]$ if $p$ is odd and  $\Mod(\Sigma_{g}^n)\slash \Mod(\Sigma_{g}^n)[2p, p/g.c.d(4,p)]$ if $p$ even. 
We prove that these quantum representations have infinite images in most cases. These results are further used in section 4 to show that 
 the orbifolds   ${\overline{\mathcal M_g^{n \; an}}}[p]$ are uniformizable, thereby proving Theorem \ref{kahler}. 
 In section 5, we explain the construction of the 
Bogomolov-Katzarkov surfaces by starting with a non isotrivial fibration of a 
smooth complex projective surface, namely the pull-back of the universal curve 
${\overline{\mathcal M_g^{1 \; an}}}$ along some holomorphic maps from
a smooth algebraic curve into ${\overline{\mathcal M_g^{an}}}$ and then taking the 
cartesian product with ${\overline{\mathcal M_g^{1 \; an}}}[p]$. We show that the Bogomolov-Katzarkov surfaces  obtained in that manner are uniformizable K\"ahler 
orbifolds, when the orbifolds ${\overline{\mathcal M_g^{1 \; an}}}[p]$ and  
${\overline{\mathcal M_g^{1 \; an}}}[2p, p/{\rm g.c.d}(4,p)]$ for even  $p$ respectively, are uniformizable.  
Section 6 contains a proof of Theorem \ref{bogo-katza} answering in the affirmative the Shafarevich conjecture on holomorphic convexity for these surfaces in most cases, 
by analyzing the images of their fundamental groups under the quantum representations above and using \cite{EKPR} (we actually only need the reductive surface case proved in \cite{KR}).

All algebraic varieties and stacks will be over $\C$ and we will mainly think of them through their analytification as complex-analytic objects. 

\vspace{0.5cm}

{\bf Acknowledgements}. The authors are indebted to J.~Aramayona, M.~Boggi, 
L.~Katzarkov, B.~Klingler, J.~March\'e and C.~Simpson for helpful discussions.

\section{The K\"ahler Orbifold ${\overline{\mathcal M_g^{n \; an}}}[p]$ and its fundamental group}
\subsection{Conventions}

As advocated by \cite{art:lerman2010}, we follow the conventions of \cite{EysOrb2} and view
an orbifold as a smooth Deligne-Mumford  stack with trivial generic isotropy groups relative to the category of complex 
analytic spaces with Hausdorff topology. There is an analytification $2$-functor  from  algebraic Deligne-Mumford stacks over $\C$ to complex analytic Deligne-Mumford stacks and an 
underlying topological stack functor from 
complex analytic Deligne-Mumford stacks to topological Deligne-Mumford stacks \cite{No1,No2,bn}. We will only consider separated quasi-compact algebraic stacks locally of finite type 
and separated analytical/topological 
stacks with proper  diagonal.

On a  complex or topological 
stack we use the Hausdorff topology,  see \cite{Stpj} for the relation with fancier topologies in the algebraic case. 
In the complex analytic and topological case,  we can always
refine a covering of  a Deligne-Mumford stack by a covering by open substacks that are quotients of a proper action of a finite group. If $\mathcal{X}$ is an algebraic 
Deligne-Mumford stack $\mathcal{X}$ there is a structural sheaf $\mathcal{O}_{\mathcal{X}}$, an abelian category of coherent algebraic sheaves and similarly in the complex analytic case, 
see \cite{LMB, CG}. 
When we represent a stack by an \'etale groupoid a sheaf comes by descent from an equivariant sheaf on the groupoid. 

An effective Cartier divisor of an algebraic/analytic is a closed substack which is defined by an invertible ideal sheaf, locally generated by a nonzero divisor. It is actually better to
slightly enlarge the notion and in our terminology a Cartier divisor on $\mathcal{X}$ will be a pair $(s,\mathcal{L})$ of an invertible sheaf $\mathcal{L}$ and a section
$s: \mathcal{O}_{\mathcal{X}} \to \mathcal{L}$ up to equivalence (multiplication by an invertible function). 
The  resulting groupoid of Cartier divisors is naturally equivalent to the groupoid ${\rm Hom}(\mathcal{X}, [\mathbb{A}^1/ \mathbb{G}_m])$ (see \cite{Ols}, p.778) where $[\mathbb{A}^1/ \mathbb{G}_m])$ is the quotient stack of the usual action of 
$\mathbb{G}_m$ on $\mathbb{A}^1$, and is an algebraic stack which is not Deligne-Mumford.

The multiplication map $m:[\mathbb{A}^1/ \mathbb{G}_m]\times [\mathbb{A}^1/ \mathbb{G}_m] \to [\mathbb{A}^1/ \mathbb{G}_m]$ enables one to define the sum of two divisors by the composition
$\mathcal{X} \to [\mathbb{A}^1/ \mathbb{G}_m]\times [\mathbb{A}^1/ \mathbb{G}_m]\buildrel{m}\over\longrightarrow [\mathbb{A}^1/ \mathbb{G}_m]$ and the pull back 
$f^*=\circ f: {\rm Hom}(\mathcal{X}, [\mathbb{A}^1/ \mathbb{G}_m]) \to {\rm Hom}(\mathcal{Y}, [\mathbb{A}^1/ \mathbb{G}_m])$, for every $f\in Ob({\rm Hom}(\mathcal{Y}, \mathcal{X}))$. 

\subsection{} 
Recall (see \cite{No2,art:lerman2010} 
that an {\em orbifold} is a smooth Deligne-Mumford stack with trivial generic isotropy 
groups relative to the category of complex analytic spaces with the classical topology. 
An orbifold $\mathcal X$ is called {\em developable} if its universal covering stack is a smooth manifold, namely if every local inertia morphism $\pi_1(\mathcal X,x)_{loc}\to \pi_1(\mathcal X,x)$
is injective, for $x$ an orbifold point of $\mathcal X$. 
The orbifold $\mathcal X$ is {\em uniformizable} whenever the profinite completion of a 
local inertia morphism, namely $\pi_1(\mathcal X,x)_{loc}\to \widehat{\pi_1(\mathcal X,x)}$
is still injective. In particular developable orbifolds with residually finite 
fundamental groups are uniformizable.

\subsection{}

For $2g-2+n>0$ the stack $\overline{\mathcal{M}_{g}^n}$ of $n$-punctured 
genus $g$ stable curves is a smooth proper Deligne-Mumford stack, whose analytification 
$\overline{\mathcal{M}_{g}^{n \; an}}$ is 
a smooth proper complex analytic Deligne-Mumford stack too (see \cite{Knu}, Thm. 2.7), actually 
an orbifold if $(g,n)\neq (2,0)$. 
Moreover, $\overline{\mathcal{M}_{g}^{n \; an}}$  carries a K\"ahler metric in the sense of \cite {EySa}, 
from Knudsen-Mumford's theorem that the moduli space $\overline{M_g^n}$ is projective 
(see \cite{Knu2}, Thm.6.1). Observe that  the forgetful map $f_1: \Mgobar\to \Mgbar$ is proper and representable (\cite{DMu}),
more precisely it is schematic (we use the convention of \cite {LMB} for representability).

\subsection{}

The stack of smooth curves 
 is realized as an open substack $j:\Mg \hookrightarrow \Mgbar$ and $\Mgbar \setminus \Mg$ is a union
of irreducible Cartier divisors $\mathcal{D}_0, \mathcal{D}_1, \ldots, \mathcal{D}_{\lfloor \frac{g}{2} \rfloor}$.

The generic point of $\mathcal{D}_0$ is a genus $g-1$ curve with one node and the generic point of 
$\mathcal{D}_k$ has two irreducible smooth
curves as components which are meeting at a node and whose respective genera are $k$ and $g-k$. 
The total divisor 
\[\mathcal{D}=\sum_{i=0}^{\lfloor \frac{g}{2} \rfloor }  \mathcal{D}_i,\] 
is a normal crossing divisor (see \cite{DMu}). 

\subsection{}

The map $j^*f_1$ is the forgetful map $\Mgone\to \Mg$ and is  representable, smooth and proper.  Hence $\mathcal{D}^1:=\Mgobar \setminus \Mgone$ is a closed substack
with a representable proper map $f_1|_{\mathcal{D}^1}: \mathcal{D}^1 \to \mathcal{D}$. Now,  
\[\mathcal{D}^1=\sum_{i=0}^{g-1} \mathcal{D}^1_{i}\]
where the generic point of $\mathcal{D}^1_{i}$ for $i>0$ represents a nodal curve with a component of genus $i$ carrying 
the one marked point and a component of genus $g-i$,   both meeting at a node which is distinct from the marked point. On the other hand the generic point of 
$\mathcal{D}^1_0$ 
represents an irreducible nodal curve of genus $g-1$ with one marked point.  
It follows that $f_1^* (\mathcal{D}_i)=\mathcal{D}_ {1,i} + \mathcal{D}_{1,g-i}$, is 
a divisor whose two components meet transversally in codimension $1$,  when $0<i<\frac{g}{2}$ whereas 
$f^*_1(\mathcal{D}_0)=\mathcal{D}^1_{0}$. 
If $g\equiv 0\; ({\rm mod}\; 2)$, then  
$f^*_1(\mathcal{D}_{\frac{g}{2}})=\mathcal{D}^1_{\frac{g}{2}}$. The divisors 
$\mathcal{D}^1_{i}$ are irreducible, the total divisor $\mathcal{D}^1$
is a normal crossing divisor (see \cite{Knu}, Thm.2.7) but not a simple normal crossing one. 
For instance $\mathcal{D}^1_{0}$ is singular at a singular point of the generic fiber of $\mathcal{D}^1_{0} \to \mathcal{D}_{0}$.

More generally, the stack ${\mathcal M}_g^n$ of smooth  $n$-pointed smooth curves of genus $g$ 
is an open substack $j:{\mathcal M}_g^n\to \overline{{\mathcal M}_g^n}$ of the 
Deligne-Mumford stack of stable $n$-pointed curves of genus $g$, when $2g-2+n>0$. 
Then $\mathcal D^n=  \overline{{\mathcal M}_g^n}\setminus {\mathcal M}_g^n$ is 
a union of irreducible Cartier divisors $\mathcal D_i$, $i=0,\ldots, N_{g,n}-1$. 
Knudsen proved in \cite{Knu} that $\overline{{\mathcal M}_g^{n+1}}$ is equivalent to the stack of  
$n$-pointed stable curves endowed with an extra section, namely the universal $n$-pointed stable curve. 
Moreover,  $\mathcal D^n=  \overline{{\mathcal M}_g^n}\setminus {\mathcal M}_g^n$ is 
a union of irreducible Cartier divisors $\mathcal D_i$, $i=0,\ldots, N_{g,n}-1$ and a normal crossings divisor. 
More specifically the projection map $f_{n+1}: \overline{{\mathcal M}_g^{n+1}}
 \to \overline{{\mathcal M}_g^n}$ has the property that 
 \[ D^{n+1}=f_{n+1}^{-1}(D^n)+ \sum_{i=1}^n S^{i, n+1},\]
 where a point of $S_{n+1}^i$ corresponds to a stable curve with a rational component 
 having two marked points labeled $i$ and $n+1$, namely $S^{i,n+1}$ is the 
 image of the canonical $i$-th section $\overline{{\mathcal M}_g^{n}}
 \to \overline{{\mathcal M}_g^{n+1}}$. 
 
\begin{lem} \label{pullback}
We have $f_{n+1}^*\mathcal{D}^n=\mathcal{D}^{n+1}- \sum_{i=1}^n  S^{i, n+1}$ and in particular, 
$f_1^*\mathcal{D}=\mathcal{D}^1$. 
\end{lem}
\begin{proof}
See (\cite{Knu}, Proof of Thm.2.7). 
\end{proof}

\subsection{The canonical stack of the root stack of a normal crossing divisor
 that may not be simple}

The first ingredient of our construction is the {\em root stack} $\mathcal{S}\left[\sqrt[p]{\mathcal{E}}\right]$ associated to a Cartier divisor  $\mathcal {E}$ on a stack  $\mathcal{S}$, 
which was introduced independently in \cite{AOV,cadman2007} and used in  \cite{ajm}.  
Specifically, it is defined as:   

\[ \mathcal{S}\left[\sqrt[p]{\mathcal{E}}\right]:= \mathcal{S} \times_{\phi_{\mathcal{E}}, [\mathbb{A}^1\slash \mathbb{G}_m] , \_^p} \left[\mathbb{A}^1\slash \mathbb{G}_m\right],\]
where 
$\phi_{\mathcal{E}}: \mathcal{S} \to \left[\mathbb{A}^1\slash \mathbb{G}_m\right]$ is the natural map and $\_^p$ is the $p$-th power map, $p\in \N^*$, the case $p=1$ being trivial. 
We use $2$-fibered products  above as one should (see \cite{LMB}). 

Note that the morphism $\mathcal{S}\left[\sqrt[p]{\mathcal{E}}\right]\to {\mathcal S}$ 
is universal among morphisms for which $\mathcal E$ pulls back to $p$ times a Cartier divisor.

The same construction can be performed with respect to
the topological and the analytic categories after application of $\_^{an}, \_^{top}$ since  these lax 2-functors preserve finite homotopy limits.

The second ingredient is the construction of the {\em canonical stack} due to Vistoli (see \cite{Vistoli}).
Recall that a Deligne-Mumford stack $\mathcal X$ over a field of characteristic zero 
has quotient singularities if it has a (locally \'etale in the algebraic case) cover 
which is a scheme with quotient singularities, namely locally the quotient of a 
smooth variety by a finite group. 
If $\mathcal{X}$ is a separated analytic  Deligne-Mumford stack  of finite
type with generic trivial isotropy having only quotient singularities 
then there is a smooth canonical stack
$\mathcal{X}^{can} \to\mathcal{X}$  which is an equivalence in codimension 1 and 
which is unique up to equivalence.

 Let $\mathcal{S}^{an}$ be an analytic smooth Deligne-Mumford stack  and $\mathcal{E}$ be a normal crossings divisor. 
 Set $\Delta^n$ for the open disk around the origin in $\C^n$ and $E_j=\{z_j=0\}\subset \Delta^n$ be the 
coordinate hyperplanes associated to a system of complex coordinates $(z_j)$. 
 An \'etale map $ \eta: \Delta^{n} \to \mathcal{S}^{an}$ is  said to be a chart adapted to  
 $\mathcal{E}$ if  
 \[\eta^*\mathcal{E}=\left\{ (z_i)_ {i=1}^n;  \  \prod_{c=1}^k z_{i_c} =0 \right\}=\sum_{c=1}^k {E}^{\eta}_{i_c},\]
where the map $c\mapsto i_c$ is injective. In order to be precise, we shall 
 at some point display the dependence of $k$ in $\eta$ and use the notation $k=k(\eta)$. 

\begin{prop}\label{rootstack}  
If $\mathcal{S}^{an}$ is an analytic smooth Deligne-Mumford stack with trivial generic isotropy and $\mathcal{E}$ is a normal crossing divisor on  $\mathcal{S}^{an}$, then there exists a smooth Deligne-Mumford stack 
$e_p:\mathcal{S}^{an}[p,\mathcal{E}] \to \mathcal{S}^{an}$ unique up to equivalence
with the same moduli space as $\mathcal{S}^{an}$ such that for every adapted chart $\eta$ we have an equivalence:
\[\Delta^n \times_{\mathcal{S}^{an}, \eta, e_p} \mathcal{S}^{an}\left[p,\mathcal{E}\right]  \simeq  \Delta^n\left[\sqrt[p]{E_{i_1}^{\eta}}\right] \times_{\Delta^n}\times \ldots
\times_{\Delta^n} \Delta^n\left[\sqrt[p]{ E_{i_k}^{\eta}}\right]. \]
\end{prop}   
\begin{proof}
We first claim that  
$\mathcal{S}^{an}\left[\sqrt[p]{\mathcal{E}_1}\right]
\times_{\mathcal{S}} \mathcal{S}\left[\sqrt[p]{\mathcal{E}_2}\right]\times_{\mathcal{S}} \cdots\times_{\mathcal{S}} \mathcal{S}\left[\sqrt[p]{\mathcal{E}_k}\right] $ has quotient singularities. 

Indeed, if one removes a codimension $\ge 2$
closed substack $\mathcal{Z}$ 
from the {\em smooth} stack 
\[\mathcal{V}=\Delta^n\left[\sqrt[p]{E_{i_1}^{\eta}}\right] \times_{\Delta^n} \cdots\times_{\Delta^n} \Delta^n\left[\sqrt[p]{ E_{i_k}^{\eta}}\right]\] 
the fundamental group will not change. Such a $\mathcal{Z}$ is actually the substack over $Z\subset \Delta^n$ where $Z$ is a closed analytic subspace. 
Then the universal covering stack of $\mathcal{V}$ is $\Delta^n$ and $\mathcal{V}=[\mu_p^k \backslash \Delta^n]$. Here $\mu_p$ is the group of $p$-th roots of unity and the 
$s$-th factor acts by multiplication of the $s$-th coordinate. 

Consider now the stack $\mathcal{U}=\Delta^n\left[\sqrt[p]{E_{I_1}}\right] \times_{\Delta^n} \cdots \times_{\Delta^n}\Delta^n \left[\sqrt[p]{E_{I_m}}\right]$, where 
we denote $E_{I}=\sum_{j\in I} E_j^{\eta}$, and $I_j$, for $1\le l \le  m$ form a partition  
of the set $\{ i_1, \ldots, i_k \}$.   
The inclusion of the residual gerbe at the origin in $\mathcal{U}$ is a deformation retract. 
Then the local uniformizability of analytic Deligne-Mumford stacks implies that
$\mathcal{U}\simeq[G\backslash U]$, where $U$ is an analytic space, $G$ has a fixed point over the origin and the natural map $m_{\mathcal U}:U\to \Delta^n$ is $G$-equivariant. Then 
$\mathcal{U}\setminus m^{-1}(Z)$ is a connected uniformizable stack \'etale over $\Delta^n\setminus Z$, in particular
$[\mathcal{U}\setminus m^{-1}(Z)]\simeq[G'\backslash(\Delta ^ n \setminus  m_{\mathcal V}^{-1}(Z)) ]$, for some finite group $G'$. 
As the root stack construction behaves well under pull back, there is 
an  isomorphism ${U}=G'\backslash \Delta^n$, where $G'=\ker (\mu_p^  k \to \mu_p^m)$ 
is the kernel of the homomorphism $(\zeta_l)\mapsto (\prod_{l \in I_1} \zeta_l,\ldots, \prod_{l \in I_m}\zeta_l)$. This proves the claim. 

Furthermore, we can define, using Vistoli's canonical stack which applies to a Deligne-Mumford stack $\mathcal{X}$ with quotient singularities and yields a smooth Deligne-Mumford stack together with a map 
$\psi:\mathcal{X}^{can} \to \mathcal{X}$ which is an equivalence in codimension one, see \cite{GerSa}:

\[ \mathcal{S}^{an}[p, {\mathcal{E}}]=\left(\mathcal{S}^{an}\left[\sqrt[p]{\mathcal{E}_1}\right]
\times_{\mathcal{S}} \mathcal{S}\left[\sqrt[p]{\mathcal{E}_2}\right]\times_{\mathcal{S}} \cdots \times_{\mathcal{S}}\mathcal{S}\left[\sqrt[p]{\mathcal{E}_k}\right]\right)^{can}.\]
The uniqueness statement is the main theorem in \cite{GerSa}  at least in the algebraic case. The proof given there can be adapted the (easier) analytic case. 

\end{proof}

Obviously, the construction does not require the ramification indices to be equal and the proof above shows that: 
\begin{prop}
Let $\mathcal{E}_1, \ldots, \mathcal{E}_r$, $1\leq r\leq \infty$, be the irreducible components of $\mathcal{E}$ and fix  $\mathbf{p}:=(p_1,\ldots, p_r)\in \mathbb{N}^{*r}$. 
There is a unique (up to equivalence) smooth Deligne-Mumford stack $\mathcal{S}^{an}\left[\mathbf{p}, \mathcal{E}\right] \to \mathcal{S}^{an}$  which is an equivalence 
outside $\mathcal{E}$ and ramifies with index $p_i$ on $\mathcal{E}_i$. 
\end{prop}
 
This is probably  valid in the algebraic  category over an algebraically closed field assuming tame ramification. 

\begin{rem} (See \cite{EySa} for a proof.) If $\mathcal{S}^{an}$ carries a K\"ahler metric, so does $\mathcal{S}^{an}\left[\mathbf{p}, \mathcal{E}\right]$. 
 
\end{rem}

\subsection{The analytic stacks $\Mgbaran[\mathbf{p}]$ }

\begin{defi} Given $\mathbf{p}\in (\N^*)^{N_{g,n}}$ we denote by  ${\overline{\mathcal{M}_{g}^{n \ an}}}[\mathbf{p}]$
the smooth proper Deligne Mumford stack  ${\overline{\mathcal{M}_{g}^{n \ an}}}[\mathbf{p},\mathcal{D}]$.
If $p, q\geq 2$ we shall use the specific notation  ${\overline{\mathcal{M}_{g}^{n \ an}}}[p,q]$ for 
the smooth proper Deligne Mumford stack  ${\overline{\mathcal{M}_{g}^{n \ an}}}[(p,q,q,\ldots,q),\mathcal{D}]$ and will drop $q$ from the notation, when $q=p$.  
\end{defi}

\begin{theo}\label{Kahlerorbifold}
 Let $g,n \in \N$ such that $2g-2+n>0$, $(g,n)\not\in\{(1,1), (2,0)\}$. 
 The quotient $\PMod(\Sigma_{g}^n)\slash \Mod(\Sigma_{g}^n)[\mathbf{p}]$  
 is the fundamental group of the compact K\"ahler orbifold
${\overline{\mathcal{M}_{g}^{n \ an}}}[\mathbf{p}]$. 
 \end{theo}

 \begin{proof}[Proof of Theorem \ref{Kahlerorbifold}]
There is an isomorphism
$\pi_1(\mathcal{M}_{g}^{n \ an})\simeq \Mod(\Sigma_{g}^{n})$ which carries the conjugate of the meridian loops around the various components of $\mathcal{D}^n$ to 
the Dehn twists along simple closed curves. 
By Proposition \ref{rootstack} the analytic stack ${\overline{\mathcal{M}_{g}^{n \ an}}}[\mathbf{p}]$
 is a proper smooth Deligne Mumford orbifold endowed with a map ${\overline{\mathcal{M}_{g}^{n \ an}}}[\mathbf{p}] \to {\overline{\mathcal{M}_{g}^{n \ an}}}$ inducing 
an isomorphism on the moduli space.
As 
${\overline{\mathcal{M}_{g}^{n \ an}}}[\mathbf{p}]$ is smooth, we can delete the codimension $2$ singular substack of $\mathcal{D}^n$ without changing  the orbifold fundamental group. 
The stack-theoretic   Van Kampen theorem  from \cite{Zoo} implies that 
the fundamental group of the  orbifold
${\overline{\mathcal{M}_{g}^{n \ an}}}[\mathbf{p}]$ is 
$\PMod(\Sigma_{g}^n)\slash \Mod(\Sigma_{g}^n)[\mathbf{p}]$.

Eventually, the orbifold ${\overline{\mathcal{M}_{g}^{n \ an}}}[\mathbf{p}]$  has projective moduli space and hence is a  K\"ahler orbifold by \cite{EySa}.  
\end{proof}
 
For the easy description of the isotropy groups of these orbifolds, see section \ref{subsect:isotropy}. 
 
\begin{exe}
The group  $\PMod(\Sigma_{0}^5)\slash \Mod(\Sigma_{0}^5)[5]$ is a  uniform complex hyperbolic group constructed by Hirzebruch:
it corresponds to the quotient of the surface of  degree $5^5$ over $\mathbb{P}^2$ which ramifies 
with multiplicity $5$ on $CEVA(2)$ by the (abelian)  group of this Galois covering. For general $p$ see \cite{ajm}. 
  \end{exe}

 A similar statement works for $\Mod(\Sigma_g^n)\slash \Mod(\Sigma_{g}^n)[p]$ taking the quotient of ${\overline{\mathcal{M}_{g}^{n \ an}}}[p]$
 by the group of permutations of $n$ punctures, see \cite{Rom} for group actions on stacks.

  \begin{rem} \label{codim2calculation}
   If we delete the singular locus $\mathcal{D}^{n, sing}$ of $\mathcal D$, the divisor of the pair $$({\overline{\mathcal{M}_{g}^{n \ an}}} \setminus \mathcal{D}^{n, sing}, \mathcal{D}^{n, reg})$$
   is actually an effective smooth divisor. The root stack $$ {\overline{\mathcal{M}_{g}^{n \ an \ 0}}}[\mathbf{p}]:= ({\overline{\mathcal{M}_{g}^{n \ an}}} \setminus \mathcal{D}^{n, sing}) [\sqrt[\mathbf{p}]{\mathcal{D}^{n, reg}} ]$$ is the complement of an 
closed substack of codimension at least $2$  in ${\overline{\mathcal{M}_{g}^{n \ an}}}[\mathbf{p}]$ hence the natural map  ${\overline{\mathcal{M}_{g}^{n \ an \ 0}}}[\mathbf{p}]\to {\overline{\mathcal{M}_{g}^{n \ an}}}[\mathbf{p}]$ 
induces an isomorphism
of the fundamental groups.

Uniformizability of $ {\overline{\mathcal{M}_{g}^{n \ an}}}[\mathbf{p}]$ can then be formulated only in terms of the moduli map $$\chi: {\overline{\mathcal{M}_{g}^{n \ an \ 0}}}[\mathbf{p}]\to {\overline{{M}_{g}^{n \ an}}}.$$ 
Indeed the isotropy morphism at (a lift of) a point $x \in {\overline{{M}_{g}^{n \ an}}}$ identifies with   the morphism $\pi_1 ( \chi^{-1}( U)) \to {\overline{\mathcal{M}_{g}^{n \ an \ 0}}}[\mathbf{p}]$ where $U\subset {\overline{{M}_{g}^{n \ an}}} $ is a small neighborhood of $x$. 
If it holds we can construct an \'etale Galois covering scheme $Z \to {\overline{\mathcal{M}_{g}^{n \ 0}}}[\mathbf{p}]$, smooth and quasiprojective, such that, if $G$ is the Galois group,   $[Z/G]={\overline{\mathcal{M}_{g}^{n \ 0}}}[\mathbf{p}]$, 
the normalization $\overline{Z}$ of ${\overline{{M}_{g}^{n}}}$ in the field of rational functions of $Z$ is smooth, the $G$ action extends to $\overline{Z}$ and $[\overline{Z} / G] \cong {\overline{\mathcal{M}_{g}^{n}}}[\mathbf{p}]$
  \end{rem}

\subsection{Orbifold compactifications and Teichm\"uller level structures}

A {\em Teichm\"uller level structure} $\lambda$ is a 
finite index normal  subgroup  $H_{\lambda} \triangleleft \PMod(\Sigma_g^n)$ (see \cite{Bry}). 
The natural finite \'etale $ G_{\lambda}:=   H_{\lambda} \backslash \PMod(\Sigma_g^n)$-covering $\mathcal{M}_g^{n, \lambda}\to \mathcal{M}_g^{n}$
can be compactified uniquely to a finite  representable $G_{\lambda}$-ramified covering $\overline{\mathcal{M}_g^{n, \lambda}}\to \overline{\mathcal{M}_g^{n}}$  (see \cite{BoPi}). 

Denote by $\overline{\mathcal{M}_g^{n}}[\lambda]$ the stack quotient $[G_{\lambda}\backslash \overline{\mathcal{M}_g^{n, \lambda}}]$ (in the sense of \cite{Rom}). 
Then $\overline{\mathcal{M}_g^{n}}[\lambda]$ is a
 Deligne-Mumford stack compactifying  ${\mathcal{M}_g^{n}}$ and dominating $\overline{\mathcal{M}_g^{n}}$. In particular, its moduli space is $\overline{M_g^n}$.

The scheme $\overline{\mathcal{M}_g^{n, \lambda}}$ or equivalently the stack $\overline{\mathcal{M}_g^{n}}[\lambda]$ may or may  not be smooth. 
The abelian level structure correspond to root stacks of $\mathcal{D}_0$ that are not smooth at the double locus of $\mathcal{D}_0$.  
For some  level structures $\lambda$,  $\overline{\mathcal{M}_g^{n}}[\lambda]$ is however smooth
(which is equivalent to  $\overline{\mathcal{M}_g^{n, \lambda}}$ being smooth)  and 
uniformizable (see \cite{Loo}  for $n=0$, and \cite{Boggi} in general),  
hence  $\overline{M^n_{g}}$ is a quotient of a smooth variety by a finite group.

Let $\mathcal{D}_g^n := \overline{\mathcal{M}_g^{n}}\setminus {\mathcal{M}_g^{n}}$.  
Then according to Knudsen (see \cite{Knu}) this is a divisor with normal crossings with an 
irreducible decomposition
\[\mathcal{D}_g^n= \sum_{i=1}^{N_{g,n}}\mathcal{D}_{g,i}^n.\]

Given a Teichm\"uller level structure  $\lambda$, we define 
$k_i(\lambda)\in  \N^*$ to be the order of the image of a Dehn twist 
corresponding to a loop encircling once $D_{g,i}^n$ in $G_{\lambda}$ and denote by $\mathbf{k}(\lambda)$ the vector $(k_i)_i=(k_i(\lambda))_i\in \N_{\ge 1}^{N_{g,n}}$.

\begin{prop} The Deligne-Mumford stack $\overline{\mathcal{M}_g^{n}}[\lambda]$ lies in a diagram of Deligne-Mumford stacks whose moduli space is $\overline{M_g^n}$ 
in which all maps are \'etale in codimension 1:
\[ \overline{\mathcal{M}_g^{n}}[\mathbf{k}(\lambda),\mathcal{D}_g^n ]  \to \overline{\mathcal{M}_g^{n}}[\lambda] \to
{\overline{\mathcal{M}_g^{n}}}\left[\sqrt[k_1]{\mathcal{D}_{g, 1}^n }\right]
\times_{\overline{\mathcal{M}_g^{n}}} {\overline{\mathcal{M}_g^{n}}} \left[\sqrt[k_2]{\mathcal{D}_{g, 2}^n }\right] \times_{\overline{\mathcal{M}_g^{n}}} \cdots 
\]
Furthermore, $ \overline{\mathcal{M}_g^{n}}[\lambda] $ is smooth if and only if  the first map is an equivalence and
$\overline{\mathcal{M}_g^{n}}[\lambda] \simeq \overline{\mathcal{M}_g^{n}}[\mathbf{k}(\lambda),\mathcal{D}_g^n ]$ 
\end{prop}
\begin{proof}
 The first statement  is clear, the second one is a consequence of the purity of the branch locus.
\end{proof}

It seems to be a rather delicate problem to understand the $\mathbf{k}(\lambda)$ that occur  and give rise to a smooth uniformizable compactification. 

\section{Quantum mapping class group representations}
\subsection{The setting of the skein TQFT}\label{tqft}

 A 3d TQFT   is, loosely speaking\footnote{See \cite{Tu} for a correct definition and a discussion of TQFT with anomalies.},  a 
functor from the category of surfaces into the category of finite dimensional vector spaces. 
Specifically, the objects of the first category are closed oriented surfaces endowed with 
colored banded points and morphisms between two objects are oriented cobordisms 
decorated by colored uni-trivalent ribbon graphs compatible with the banded points.
A banded point on a surface is a point with a tangent vector at that point, or equivalently 
a germ of an oriented interval embedded in the surface. There is a corresponding  
surface with colored boundary obtained by deleting a small neighborhood of the 
banded points and letting the boundary circles inherit the colors of the respective points.

We will use the TQFT functor $\mathcal V_p$, 
for $p\geq 3$ and a primitive root of unity $A$ of order $2p$, as 
defined in \cite{BHMV}. 
The vector space  associated by the functor $\mathcal V_p$ to a surface 
is called the {\em space of conformal blocks}.  Let $\Sigma_g$ denote the genus $g$ closed orientable surface, 
$H_g$ be a genus $g$ handlebody with $\partial H_g=\Sigma_g$. 
 Assume given a finite set  $\mathcal Y$ of banded 
points on $\Sigma_g$. Let $G$ be a uni-trivalent ribbon graph 
embedded in $H_g$ in such a way that 
$H_g$ retracts onto $G$, its univalent vertices are the banded points $\mathcal Y$ and it has no other intersections 
with $\Sigma_g$.

The natural number $p\geq 3$ is called the {\em level} of the TQFT. 
We define the {\em set of colors} in level $p$ to be 
${\mathcal C}_p=\{0,2,4,\ldots,p-3\}$, if $p$ is odd and 
${\mathcal C}_p=\{0,1,2,\ldots,\frac{p-4}{2}\}$, if $p$ is even, respectively. 
 
An edge coloring of $G$ is called {\em{$p$-admissible}} if the following conditions are satisfied: 

\begin{enumerate}
\item The {\em triangle inequality} is satisfied at each trivalent vertex of $G$. Namely,  if $i,j,k$ are the colors of the incoming edges,  one has $|j-k|\le i\le j+k$. 
\item The sum of of the three colors around a vertex is even, this condition being empty for odd $p$.
\item The sum of  the three colors around a vertex is at most $q=2p-4$, if $p$ is odd and 
at most $q=p-4$, if $p$ is even, respectively. We call it the {\em $q$-bound inequality}. 
\end{enumerate}

Fix a coloring of the banded points $\mathcal Y$. Then there exists a basis of the space of conformal blocks associated to the surface $(\Sigma_g, \mathcal Y)$ with the  colored banded points (or the corresponding surface with colored boundary) 
which is indexed by the set of all $p$-admissible colorings of $G$ 
extending the boundary coloring. We  denote by $W_{g, p, (i_1,i_2,\ldots,i_r)}$ (or, dropping $p$ from the notation, by $W_{g, (i_1,i_2,\ldots,i_r)}$, 
if $p$ is uniquely determined by the context)  the vector space associated to the closed surface $\Sigma_g$ with $r$ banded points colored by $i_1,i_2,\ldots,i_r\in \mathcal C_p$. Note that 
banded points colored by $0$ do not contribute. 

Observe that an admissible $p$-coloring of $G$ provides an element of the skein module \cite{Prz} 
$S_{A}(H_g, \mathcal Y, (i_1,i_2,\ldots,i_r))$ \footnote{If the context allows, we will use the shorthand notations $S_{A}(H_g,   \mathcal Y, (i_1,i_2,\ldots,i_r))=S_{A}(H_g, \mathcal Y)=
S_{A}(H_g)$.} of the handlebody  with the banded boundary points $\mathcal Y$ colored by $(i_1,i_2,\ldots,i_r)$, evaluated at the  
primitive $2p$-th root of unity $A$. This skein element is obtained by cabling the edges of $G$ by the Jones-Wenzl 
idempotents  prescribed by the coloring and having banded points colors fixed.  Jones-Wenzl idempotents were discovered by V. Jones \cite{Jo}
and their inductive construction is due to Wenzl \cite{We}.

We suppose that $H_g$ is embedded in a standard way into the $3$-sphere $S^3$, so that the closure of 
its complement is also a genus $g$ handlebody $\overline{H}_g$.  There is then a 
sesquilinear form: 
\[ \langle \;,\; \rangle=\langle \;,\; \rangle_p: S_{A}(H_g, \mathcal Y, (i_1,i_2,\ldots,i_r))\times S_{A}(\overline{H}_g, - \mathcal Y, (i_1,i_2,\ldots,i_r))\to \C\]
defined by 
\[ \langle x, y \rangle= \langle x \sqcup y \rangle.\]
Here $x\sqcup y$ is the element of $S_{A}(S^3)$ is the union of  $x$ and $y$ along $\mathcal Y$ in 
$H_g\cup\overline{H}_g=S^3$, and $\langle \; \rangle: S_{A}(S^3)\to \C$ is the Kauffman bracket invariant \cite{Ka}. 

The space of conformal blocks $W_{g,(i_1,i_2,\ldots,i_r)}$ is the quotient 
$S_{A}(H_g)/\ker \langle\;,\; \rangle$ by the left kernel of the sesquilinear form above. It follows that $W_{g,(i_1,i_2,\ldots,i_r)}$ is endowed with
an induced {\em Hermitian form} $H_{A}$. 

The projections of skein elements associated 
to the $p$-admissible colorings of a trivalent graph $G$ as above form an orthogonal basis of $W_{g,(i_1,i_2,\ldots,i_r)}$ with respect to $H_{A}$. It is known (\cite{BHMV}) 
that $H_{A}$ only depends on the $p$-th root of unity $\zeta_p=A^2$ and that in this orthogonal basis the diagonal entries belong to the totally real maximal subfield $\Q(\zeta_p+\overline{\zeta_p})$ (after rescaling). 
If $\sigma$ is a $p$-admissible coloring of $G$ extending the boundary coloring, then  
the diagonal term of $H_A$ associated to this basis vector is 
\begin{equation}\label{hermitian}
\eta^{g+1} \frac{\prod_{v\; {\rm vertex}}\langle v\rangle}{\prod_{e \; {\rm edge}}\langle e\rangle},
\end{equation}
where products are taken over the set of all trivalent vertices $v$ and all edges $e$ of the graph $G$,
and 
\[ \eta=\frac{1}{2p}(A\kappa)^3(A^2-A^{-2})\sum_{j=1}^{2p}(-1)^j A^{j^2}, \]
where 
\[\kappa^6= A^{-6-\frac{p(p+1)}{2}}\]
is a fixed parameter of the theory. Note that changing $\kappa$ to $-\kappa$ will change the signature of $H_A$ into its opposite. 
For an edge $e$ labeled by the color $a$ we put 
\[ \langle e \rangle = (-1)^a [a+1],\]
and for a trivalent vertex whose adjacent edges are labeled in counterclockwise order 
by the colors $a,b,c$ we put 
\begin{equation} \langle v \rangle = (-1)^{i+j+k+1}\frac{[i+j+k+1]![i]![j]![k]!}{[a]![b]![c]!},\end{equation}
where 
\begin{equation}
i=\frac{a+b-c}{2}, \; j=\frac{b+c-a}{2}, \; k=\frac{c+a-b}{2}, 
\end{equation}
\begin{equation}
[m]=\frac{A^{2m}-A^{-2m}}{A^2-A^{-2}}, \; [m]!=[1][2]\cdots [m-1][m].
\end{equation}

Let $G'\subset G$ be a uni-trivalent subgraph  whose degree one vertices are colored, corresponding to a 
sub-surface $\Sigma'$ of $\Sigma_g$ with colored boundary. The projections in $W_{g,(i_1,i_2,\ldots,i_r)}$ of skein elements associated to 
the $p$-admissible colorings of $G'$ form an orthogonal basis of the space of conformal blocks 
associated to the surface $\Sigma'$ with colored boundary components. 

There is a linear geometric action of the mapping class groups of the handlebodies with marked banded boundary $(H_g, \mathcal{Y})$ and $(\overline{H}_g, - \mathcal{Y})$ respectively
on their skein modules and hence on the space of conformal blocks. Moreover, these actions extend to a projective  
action $\rho_{g, p, (i_1,\ldots,i_r),A}$ of $\Mod(\Sigma_g^r)$ on $W_{g,(i_1,i_2,\ldots,i_r)}$ respecting 
the Hermitian form $H_{\zeta_p}=H_A$. When referring to $\rho_{g, p, (i_1,\ldots,i_r),A}$ the subscript 
specifying the genus $g$ will most often be dropped when its value will be clear from the context.
Notice that the mapping class group of an essential (i.e. without annuli or disks complements) 
sub-surface $\Sigma'\subset \Sigma_g$ is a subgroup of $\Mod(\Sigma_g)$ which preserves the subspace of conformal blocs 
associated to $\Sigma'$ with colored boundary.  It is worthy to note that  $\rho_{p,(i_1,\ldots,i_r), A}$ only depends on 
$\zeta_p=A^2$, so we can unambigously shift the notation for this representation to $\rho_{p,(i_1,\ldots,i_r),\zeta_p}$. We will often drop the reference to the root of unity $\zeta_p$, as changing it 
amounts to use a Galois conjugacy of the real cyclotomic field.  
 
There is a central extension $\widetilde{\Mod(\Sigma_g)}$ of $\Mod(\Sigma_g)$ by $\Z$ and a linear representation  
 $\ro_{p, \zeta_p}$ on $W_{g}$ which  resolves the projective ambiguity of $\rho_{p,\zeta_p}$. 
 The largest such central extension has class $12$ times the Euler class (see \cite{Ger,MR}), but 
 the central extension considered in this paper is  an index $12$ subgroup of it, called $\widetilde{\Gamma}_1$ 
 in \cite{MR}.

 We denote by $\Sigma_{g,n}^r$ the compact orientable surface of genus $g$ with $n$ boundary components and $r$ marked points. Then $\Mod(\Sigma_{g,n}^r)$ denotes the pure mapping class group of $\Sigma_{g,n}^r$ which fixes pointwise boundary components and marked points. 
 
We consider a subsurface $\Sigma_{g,r}\subset \Sigma_{g+r}$ whose complement consists of $r$ copies of $\Sigma_{1,1}$. 
Let $\widetilde{Mod(\Sigma_{g,r})}$ be the pull-back of the central extension $\widetilde{\Mod(\Sigma_g)}$ to 
the subgroup $\Mod(\Sigma_{g,r})\subset \Mod(\Sigma_{g+r})$. Then $\widetilde{\Mod(\Sigma_{g,r})}$ is also a central extension, which we denote 
$\widetilde{\Mod(\Sigma_{g}^r)}$, of $\Mod(\Sigma_g^r)$ by $\Z^{r+1}$.

\begin{definition}\label{qrep}
Let $p\geq 4$ and $\zeta_p$  a primitive 
$p$-th root of unity. We denote 
by $\ro_{p, \zeta_p, (i_1,i_2,\ldots,i_r)}$ the linear representation of the central extension 
$\widetilde{\Mod(\Sigma_{g}^r)}$ which acts on the vector space $W_{g,p,(i_1,i_2,\ldots,i_r)}$ associated by the
TQFT to the surface with the corresponding colored banded points (see \cite{Ger,MR}).  
\end{definition}

The functor $\mathcal V_p$ associates to a handlebody $H_g$ the projection of the skein element 
corresponding to the trivial coloring of the trivalent graph $G$ by $0$. The invariant associated 
to a closed 3-manifold is given by pairing the two vectors associated to handlebodies in a Heegaard decomposition 
of some genus $g$ and taking into account the twisting by the gluing mapping class action on $W_g$. 

One should notice that the skein TQFT $\mathcal V_p$ is unitary, in the sense that $H_{\zeta_p}$ is 
a positive definite Hermitian form when $\zeta_p=(-1)^p \exp\left(\frac{2\pi i}{p}\right)$, corresponding to  
$A_p=(-1)^{\frac{p-1}{2}}\exp\left(\frac{(p+1)\pi i}{2p}\right)$. 
For the sake of notational simplicity, from now we will drop the subscript $p$ 
in $\zeta_p$, when the order of the root of unity will be clear from the context and the precise choice of the root of given order won't matter.  
Note that for a general primitive $p$-th root of unity, 
the isometries of $H_{\zeta}$ form a pseudo-unitary group.

Now, the image $\rho_{p,\zeta}(T_{\gamma})$  of a right hand Dehn twist $T_{\gamma}$  in 
a convenient basis given by a trivalent graph is easy to describe. 
Assume that the simple curve $\gamma$ is the boundary of a small disk intersecting once transversely 
an edge  $e$ of the graph $G$. Consider $v\in W_g$ be a vector of the basis given by colorings of the graph $G$ 
and assume that edge $e$ is labeled by the color $c(e)\in \mathcal C_p$. Then the action of 
the (canonical)  lift  $\widetilde{T_{\gamma}}$ of the Dehn twist $T_{\gamma}$ in $\widetilde{\Mod(\Sigma_g)}$ 
is given by (see \cite{BHMV}, 5.8) :
\[ \ro_{p,\zeta}(\widetilde{T_{\gamma}}) v =(-1)^{c(e)}A^{c(e)(c(e)+2)} v.\]

Actually the central extension $1 \to \mathbb{Z} \to \widetilde{\Mod(\Sigma_g)}\to \Mod(\Sigma_g) \to 1 $ is trivial over
$\Mod(H_g) <  \Mod(\Sigma_g)$ which enables to define $\ro_{p,\zeta}(\widetilde{\sigma})$ for every mapping class $\sigma\in \Mod(H_g)$
as the geometric action of $\sigma$ on the reduced skein module of $H_g$. 
  The condition on $\gamma$ implies that $T_{\gamma} \in \Mod(H_g)$.

\subsection{Spaces of conformal blocks in genus zero}

In the rest of the article, $p$ will be a positive integer, $n$ a non negative integer, $i_1, \ldots, i_n$ a choice of $n$ colors
in $\mathcal{C}_p$ and $\rho_{p, (i_1, \ldots, i_n)}$ 
will denote the homomorphism of $\Mod(\Sigma_g^n)$  in the {\em projective} general linear group
of the space of conformal blocks deduced from $\tilde \rho_p$. 
We use the convention that the color $0$ be dropped from the boundary colors, as the corresponding representations are the same. 

A simple closed curve on a surface is {\em essential} if it is neither null-homotopic nor homotopic to a 
boundary component. Essential simple closed curves  are either called {\em non-separating}, 
when their complement is connected, or {\em separating}, otherwise. The {\em genus} of a separating 
simple closed curve is the minimum of the genera of the two connected subsurfaces 
which it bounds. 

\begin{defi}
Define $\delta_p(\mathbf i)$, where $\mathbf i=(i_1,i_2,\ldots,i_n)$, 
$i_1,i_2,\ldots,i_n\in \mathcal C_p$, as 
the number of those $j\in \mathcal C_p$ such that 
$\dim W_{0,p,(i_1,i_2,\ldots,i_n,j)}\neq 0$. 
The $n$-tuple $\mathbf i=(i_1,i_2,\ldots,i_n)$ is  said to be {\em generic} when 
$\delta_p(\mathbf i)\geq 3$. 
\end{defi}

The aim of this section is to give explicit formulas for $\delta_p(i_1,i_2,\ldots,i_n)$.

\begin{lem}\label{lemma:null}
Let $p\geq 5$ be odd. 
The space of conformal blocks $W_{0,p,(i_1,i_2,\ldots,i_m)}$, $m\geq 3$, the colors 
$i_1,i_2,\ldots,i_m$ not all  equal to $0$, is non-zero if and only if the condition $P(k,m)$ below is satisfied:
\begin{equation}\label{P} 2 \sum_{\alpha=1}^{2k+1} i_{s_{\alpha}} \leq 2k(p-2)+ \sum_{t=1}^m i_t, \end{equation} 
for all $0\leq k\leq \frac{m-1}{2}$ and  all subsets $\{s_1,s_2,\ldots,s_{2k+1}\}\subset\{1,2,\ldots,m\}$.
\end{lem}
\begin{proof}
These conditions for $k=0$  (triangle inequalities) and $k=1$  (the $q$-bound inequality) 
are equivalent to the $p$-admissibility, in the first non-trivial case $m=3$. 

Consider now a pants decomposition of the holed sphere, every pair of pants having one 
boundary curve from $\partial \Sigma_{0,m}$ except two of them each of which has precisely two boundary
curves from $\partial \Sigma_{0,m}$.  We choose a subset of $k$ vertices of the dual graph  
such that their stars are disjoint. We sum the $q$-bound inequalities  corresponding to these $k$ vertices 
with the triangle inequalities of the remaining vertices, which concern the lower bounds for the 
internal edges, and obtain the inequality $P(k,m)$.

Conversely, we claim that the $P(k,m)$ conditions are sufficient to guarantee the existence of 
a $p$-admissible coloring of the dual graph. 
We proceed by induction on $m$. Suppose the claim is true for all spheres with at most $m$ boundary components and consider $(i_1,i_2,\ldots,i_{m+1})$ satisfying the conditions $P(k,m+1)$. 
We want to find $j_m\in \mathcal C_p$ such that the triple $(j_m,i_m,i_{m+1})$ satisfies the $p$-admissibility 
condition and $(i_1,i_2,\ldots,i_{m-1},j_m)$ satisfies the conditions $P(k,m)$. 
This amounts to the following two sets of inequalities: 
\begin{equation}\label{i2}
 |i_{m+1}-i_m| \leq j_m\leq \min(i_m+i_{m+1}, 2(p-2)-(i_m+i_{m+1}),p-3),
 \end{equation}
\begin{equation}\label{i3}
2 \sum_{\alpha=1}^{2k+1} i_{s_{\alpha}} -2k(p-2)-\sum_{t=1}^{m-1} i_t
\leq j_m \leq 2k(p-2)+\sum_{t=1}^{m-1} i_t -2\sum_{\alpha=1}^{2k} i_{s_{\alpha}}.
\end{equation}
These inequalities have solutions $j_m\in \mathcal C_p$ if and only if the intervals defined above have non-empty intersection, namely if 
\begin{equation}\label{i4}
 |i_{m+1}-i_m| \leq 2k(p-2)+\sum_{t=1}^{m-1} i_t -2\sum_{\alpha=1}^{2k} i_{s_{\alpha}}\end{equation}
and 
\begin{equation}\label{i5} 
 2 \sum_{\alpha=1}^{2k+1} i_{s_{\alpha}} -2k(p-2)-\sum_{t=1}^{m-1} i_t \leq
\min(i_m+i_{m+1}, 2(p-2)-(i_m+i_{m+1}),p-3).\end{equation}
Inequality (\ref{i4}) is equivalent to $P(k,m+1)$. 
Inequality (\ref{i5}) consists of three inequalities, which are respectively equivalent to $P(k,m+1)$, $P(k+1,m+1)$ 
and 
\begin{equation}
2\sum_{\alpha=1}^{2k+1} i_{s_{\alpha}} \leq 2k(p-2)+(p-3)+\sum_{t=1}^{m-1} i_t.
\end{equation}  
This is a consequence of the two former inequalities and the fact that all $i_j$ are even. 
\end{proof}

\begin{lem}\label{lemma:deltaodd}
Let $p\geq 5$ be odd. Then for every  $n\geq 2$, $\mathbf i=(i_1,i_2,\ldots,i_n)$, $i_1,i_2,\ldots,i_n\in \mathcal C_p$ 
such that $i_1\leq i_2\leq \cdots \leq i_n$ \footnote{There is no loss of generality in assuming 
$i_1\leq i_2 \leq \cdots \leq i_m$.}, we have: 
\[ \delta_p(\mathbf i)= 1+\frac{1}{2}\left(J_{p, \max}(\mathbf i)-J_{p, \min}(\mathbf i)\right),\]
where 
\begin{equation}
J_{p, \max}(\mathbf i)=\min_{0\leq \ell\leq \frac{n}{2}}\min\left(p-3,\sum_{t=1}^{n-2\ell}i_t-\sum_{s=n-2\ell +1}^ni_{s} +2\ell(p-2)\right), 
\end{equation}
\begin{equation}
J_{p, \min}(\mathbf i)=\max_{0\leq k\leq \frac{n}{2}}\max\left(0,\sum_{s=n-2k}^ni_{s} -\sum_{t=1}^{n-2k-1}i_t-2k(p-2)\right).
\end{equation}
In particular, $\delta_p(\mathbf i)\geq 1$. 
\end{lem}
 
\begin{proof}
From Lemma \ref{lemma:null} it suffices to analyze the solutions $j\in \mathcal C_p$ of the system of inequalities $P(k,n+1)$
and prove they are all even colors $j$ such that $J_{p,\min}\le j \le J_{p,\max}$ observing that $J_{p,\min}\equiv 0 \; ({\rm mod}\; 2), \ J_{p, \max}\equiv 0 \; ({\rm mod}\; 2)$. 
 
This is equivalent to finding $j\in \mathcal C_p$ such 
that: 
\begin{equation}
\sum_{s=n-2k}^ni_{s} -\sum_{t=1}^{n-2k-1}i_t-2k(p-2)\leq 
j \leq \sum_{t=1}^{n-2\ell}i_t-\sum_{s=n-2\ell +1}^ni_{s} +2\ell(p-2),
\end{equation}
for all $0\le k \le \frac{n}{2}, 0 \le l \le \frac{n}{2}$.
This system has solutions if and only if for every $k, \ell$ we have 
\begin{equation}\label{ineq1}
\max\left(0,\sum_{s=n-2k}^ni_{s} -\sum_{t=1}^{n-2k-1}i_t-2k(p-2)\right)\leq 
\min\left(p-3,\sum_{t=1}^{n-2\ell}i_t-\sum_{s=n-2\ell +1}^ni_{s} +2\ell(p-2)\right).
\end{equation}

Consider one of the inequalities involved in (\ref{ineq1}), say: 
\[\sum_{s=n-2k}^ni_{s} -\sum_{t=1}^{n-2k-1}i_t-2k(p-2)\leq 
\sum_{t=1}^{n-2\ell}i_t-\sum_{s=n-2\ell +1}^ni_{s} +2\ell(p-2),\]
which reads:
\[\sum_{s=n-2k}^ni_{s}+ \sum_{s=n-2\ell +1}^ni_{s} \leq 
\sum_{t=1}^{n-2\ell}i_t +\sum_{t=1}^{n-2k-1}i_t+2(k+\ell)(p-2).\]
For $k+1\leq \ell$ this is equivalent to:
\[ 2\sum_{s=n-2k}^ni_{s}\leq 
2\sum_{t=1}^{n-2\ell}i_t +2(k+\ell)(p-2).\]
This inequality follows from:
\[ 2\sum_{s=n-2k}^ni_{s} < 2 (2k+1)(p-2)\leq 
2\sum_{t=1}^{n-2\ell}i_t +2(k+\ell)(p-2)\]
and equality cannot occur. 
Further, if $k\geq \ell$ this amounts to 
\[ 2\sum_{s=n-2\ell +1}^ni_{s}\leq 
2\sum_{t=1}^{n-2k-1}i_t +2(k+\ell)(p-2),\]
which follows again from
\[ 2\sum_{s=n-2\ell +1}^ni_{s}\leq 4\ell (p-2) 
\leq 2\sum_{t=1}^{n-2k-1}i_t +2(k+\ell)(p-2)\]
with equality only if $k=\ell=0$ and 
\[ i_1=i_2=\cdots i_{n-1}=0.\]
Similar arguments show that all  inequalities involved  in (\ref{ineq1}) are valid for all $i_s\in \mathcal C_p$
and hence $\delta_p(i_1,i_2,\ldots,i_n)\geq 1$. 
\end{proof}

We will need later also the following related function: 
\begin{defi}
We let $\delta^{(1)}_p(\mathbf i)$,  $\mathbf i=(i_1,i_2,\ldots,i_n)$, $i_1,i_2,\ldots,i_n\in \mathcal C_p$,  be the number of 
those $j\in \mathcal C_p$ with the property that 
$\dim W_{0,p,(i_1,i_2,\ldots,i_n,j,j)}\neq 0$. 
\end{defi}

\begin{lem}\label{lemma:deltaoddg1}
Let $p\geq 5$ be odd. Then for every  $n\geq 1$, $\mathbf i=(i_1,i_2,\ldots,i_n)$, $i_s\in \mathcal C_p$ 
such that $i_1\leq i_2\leq \cdots \leq i_n$, we have 
\[ \delta^{(1)}_p(\mathbf i)= 1+
\left\lfloor\frac{J^{(1)}_{p, \max}(\mathbf i)}{4}\right\rfloor-\left\lceil\frac{J_{p, \min}(\mathbf i)}{4}\right\rceil,
\]
where
\begin{equation}
J^{(1)}_{p, \max}(\mathbf i)=\min_{1\leq \ell\leq \frac{n+1}{2}}\min\left(2(p-3),\sum_{t=1}^{n-2\ell+1}i_t-\sum_{s=n-2\ell +2}^ni_{s} +2\ell(p-2)\right),
\end{equation}
\begin{equation}
J^{(1)}_{p, \min}(\mathbf i)=\max_{0\leq k\leq \frac{n-1}{2}}\max\left(0,\sum_{s=n-2k}^ni_{s} -\sum_{t=1}^{n-2k-1}i_t-2k(p-2)\right).
\end{equation}
\end{lem}
\begin{proof}
The claim follows by direct computation from  Lemma \ref{lemma:deltaodd},  reducing us to count the number of even solutions $j\in \mathcal C_p$ 
of:  
\[  J^{(1)}_{p, \min}(\mathbf i)\leq 2j \leq J^{(1)}_{p, \max}(\mathbf i).\]
 
\end{proof}

\begin{rem} Observe we always have $ \delta^{(1)}_p(\mathbf i) \ge 1$ if $\mathbf{i}\not=(0, \ldots, 0)$. This follows from $\delta_p(i_1, \ldots, i_n, i_n)\ge 1$.
 It is not completely obvious from the formulas given above that 
 \[ \left\lfloor\frac{J^{(1)}_{p, \max}(\mathbf i)}{4}\right\rfloor \ge \left\lceil\frac{J_{p, \min}(\mathbf i)}{4}\right\rceil. \]
\end{rem}

\begin{lem}\label{lemma:nulleven}
Let $p\geq 6$ be even. 
The space of conformal blocks $W_{0,p,(i_1,i_2,\ldots,i_m)}$, $m\geq 3$, the colors 
$i_1,i_2,\ldots,i_m$ not all  equal to $0$, is non-zero if and only if, first the 
following parity obstruction  
\begin{equation}\label{jparity}
\sum_{t=1}^m i_t\equiv 0 \; ({\rm mod}\; 2)
\end{equation}
holds and 
second the condition $P(k,m)$ below is
satisfied: 
\begin{equation}\label{Peven} 2 \sum_{\alpha=1}^{2k+1} i_{s_{\alpha}} \leq k(p-4)+ \sum_{t=1}^m i_t \end{equation} 
for all $0\leq k\leq \frac{m-1}{2}$ and  all subsets $\{s_1,s_2,\ldots,s_{2k+1}\}\subset\{1,2,\ldots,m\}$.
\end{lem}
\begin{proof}
The proof follows the pattern of the corresponding result for odd $p$. 
The parity obstruction, and the $P(k,m)$  condition for $k=0$  (triangle inequalities) and $k=1$  (the $q$-bound inequality) are equivalent to the $p$-admissibility, in the first non-trivial case $m=3$. 

Consider now a pants decomposition of the holed sphere, every pair of pants having one 
boundary curve from $\partial \Sigma_{0,m}$ except two of them each of which has precisely two boundary
curves from $\partial \Sigma_{0,m}$.  Summing up the parity obstruction for vertices we obtain the 
parity obstruction for arbitrary $m$. The $P(k,m)$ condition is proved as above.

Conversely, we claim that the $P(k,m)$ conditions are sufficient to guarantee the existence of 
a $p$-admissible coloring of the dual graph. 
We proceed by induction on $m$. Suppose the claim is true for all spheres with at most $m$ boundary components and consider $(i_1,i_2,\ldots,i_{m+1})$ satisfying the conditions $P(k,m+1)$. 
We want to find $j_m\in \mathcal C_p$ such that the triple $(j_m,i_m,i_{m+1})$ satisfies the $p$-admissibility 
condition and $(i_1,i_2,\ldots,i_{m-1},j_m)$ satisfies the conditions $P(k,m)$. 
This amounts to one parity obstruction 
\begin{equation}\label{parity1}
i_m+i_{m+1}+j_m\equiv 0 \; ({\rm mod} \; 2),
\end{equation}
along with the following two sets of inequalities: 
\begin{equation}\label{i61}
 |i_{m+1}-i_m| \leq j_m\leq \min\left(i_m+i_{m+1}, p-4-(i_m+i_{m+1}),\frac{p-4}{2}\right),
\end{equation}
\begin{equation}\label{i71}
2 \sum_{\alpha=1}^{2k+1} i_{s_{\alpha}} -k(p-4)-\sum_{t=1}^{m-1} i_t
\leq j_m \leq k(p-4)+\sum_{t=1}^{m-1} i_t -2\sum_{\alpha=1}^{2k} i_{s_{\alpha}}.
\end{equation}
These inequalities have solutions $j_m\in \mathcal C_p$ if and only if
\begin{equation}\label{i6}
 |i_{m+1}-i_m| \leq k(p-4)+\sum_{t=1}^{m-1} i_t -2\sum_{\alpha=1}^{2k} i_{s_{\alpha}}\end{equation}
and 
\begin{equation}\label{i7} 
 2 \sum_{\alpha=1}^{2k+1} i_{s_{\alpha}} -k(p-4)-\sum_{t=1}^{m-1} i_t \leq
\min\left(i_m+i_{m+1}, p-4-(i_m+i_{m+1}),\frac{p-4}{2}\right).\end{equation}
Inequality (\ref{i6}) is equivalent to $P(k,m+1)$. 
Inequality (\ref{i7}) consists of three inequalities, which respectively equivalent to $P(k,m+1)$, $P(k+1,m+1)$ 
and 
\begin{equation}
2\sum_{\alpha=1}^{2k+1} i_{s_{\alpha}} \leq \left(k+\frac{1}{2}\right)(p-4)+\sum_{t=1}^{m-1} i_t.
\end{equation}  
This is a consequence of the former two inequalities. 

It remains to observe that there exist solutions satisfying the parity condition above. 
In fact both endpoints of the interval specified by inequality (\ref{i61}) are compatible 
with the parity obstruction (\ref{parity1}), while the endpoints of the interval given by 
(\ref{i71}) are both congruent to $i_m+i_{m+1} \; ({\rm mod} \;  2)$. It follows that there exists a solution $j_m$ satisfying the parity obstruction. 
\end{proof}

\begin{lem}\label{lemma:deltaeven}
Let $p\geq 6$ be even. Then for every  $n\geq 2$, $\mathbf i=(i_1,i_2,\ldots,i_n)$,  $i_1,i_2,\ldots,i_n\in \mathcal C_p$ 
such that $i_1\leq i_2\leq \cdots \leq i_n$, we have 
\[ \delta_p(\mathbf i)= 1+\frac{1}{2}\left(J_{p, \max}(\mathbf i)-J_{p, \min}(\mathbf i)\right),\]
where 
\begin{equation}
J_{p, \max}(\mathbf i)=\min_{0\leq \ell\leq \frac{n}{2}}\min\left(\frac{p-4}{2}-\varepsilon_{\mathbf i}(p),\sum_{t=1}^{n-2\ell}i_t-\sum_{s=n-2\ell +1}^ni_{s} +\ell(p-4)\right),
\end{equation}
\begin{equation}
J_{p,\min}(\mathbf i)=\max_{0\leq k\leq \frac{n}{2}}\max\left(\varepsilon_{\mathbf i},\sum_{s=n-2k}^mi_{s} -\sum_{t=1}^{n-2k-1}i_t-k(p-4)\right),
\end{equation}
\begin{equation}
\varepsilon_{\mathbf i}\in \{0,1\} \; {\rm such \;  that  } \; \varepsilon_{\mathbf i}\equiv \sum_{s=1}^ni_{s} \; ({\rm mod}\; 2),
\end{equation}
\begin{equation}
\varepsilon_{\mathbf i}(p)=\left\{\begin{array}{ll}
\varepsilon_{\mathbf i}, & \; {\rm if  } \; p \equiv 0 \; ({\rm mod }\; 4);\\
1-\varepsilon_{\mathbf i}, &  \; {\rm if  } \; p \equiv 2 \; ({\rm mod }\; 4).
\end{array}\right.
 \end{equation}
In particular, $\delta_p(\mathbf i)\geq 1$. 
\end{lem}
\begin{proof}
From Lemma \ref{lemma:nulleven} it suffices to analyze the solutions $j\in \mathcal C_p$ 
of the system of inequalities $P(k,n+1)$ with the additional parity constraint (\ref{jparity}). There is no loss of generality in assuming that 
$i_1\leq i_2 \leq \cdots \leq i_n$. Therefore our system is equivalent to finding $j\in \mathcal C_p$ such 
that: 
\begin{equation}
\sum_{s=m-2k}^ni_{s} -\sum_{t=1}^{n-2k-1}i_t-k(p-4)\leq 
j \leq \sum_{t=1}^{n-2\ell}i_t-\sum_{s=n-2\ell +1}^mi_{s} +\ell(p-4).
\end{equation}
This system has solutions if and only if for every $k, \ell$ we have 
\begin{equation}\label{ineq2}
\max\left(0,\sum_{s=n-2k}^mi_{s} -\sum_{t=1}^{n-2k-1}i_t-k(p-4)\right)\leq 
\min\left(\frac{p-4}{2},\sum_{t=1}^{n-2\ell}i_t-\sum_{s=n-2\ell +1}^ni_{s} +\ell(p-4)\right).
\end{equation}
Moreover, we have solutions $j$ which also satisfy the parity condition (\ref{jparity}) if and only if 
\[ J_{p,\min}(\mathbf i)\leq J_{p,\max}(\mathbf i)\]
and then $\delta_p(i_1,i_2,\ldots,i_n)$ is the number of integers 
in this interval congruent to $\varepsilon_{\mathbf i}$ modulo $2$, as claimed.

Consider one of the inequalities involved in (\ref{ineq2}), say: 
\[\sum_{s=n-2k}^ni_{s} -\sum_{t=1}^{n-2k-1}i_t-k(p-4)\leq 
\sum_{t=1}^{n-2\ell}i_t-\sum_{s=n-2\ell +1}^ni_{s} +\ell(p-4),\]
which reads:
\[\sum_{s=n-2k}^mi_{s}+ \sum_{s=n-2\ell +1}^ni_{s} \leq 
\sum_{t=1}^{n-2\ell}i_t +\sum_{t=1}^{n-2k-1}i_t+(k+\ell)(p-4).\]
For $k+1\leq \ell$ this is equivalent to:
\[ 2\sum_{s=n-2k}^ni_{s}\leq 
2\sum_{t=1}^{n-2\ell}i_t +(k+\ell)(p-4).\]
This inequality follows from:
\[ 2\sum_{s=n-2k}^ni_{s} \leq (2k+1)(p-4)\leq 
2\sum_{t=1}^{n-2\ell}i_t +2(k+\ell)(p-4).\]
Further, if $k\geq \ell$ this amounts to 
\[ 2\sum_{s=n-2\ell +1}^ni_{s}\leq 
2\sum_{t=1}^{n-2k-1}i_t +(k+\ell)(p-4),\]
which follows again from
\[ 2\sum_{s=n-2\ell +1}^ni_{s}\leq 2\ell (p-4) 
\leq 2\sum_{t=1}^{n-2k-1}i_t +(k+\ell)(p-4)\]
with equality only if $k=\ell=0$ and 
\[ i_1=i_2=\cdots i_{n-1}=0.\]
Similar arguments show that all  inequalities involved  in (\ref{ineq2}) are valid for all $i_s\in \mathcal C_p$
and hence $\delta_p(\mathbf i)\geq 1$. 
\end{proof}
 
\begin{lem}\label{lemma:deltaeveng1}
Let $p\geq 6$ be even. Then for every   $n\geq 1$, $\mathbf i=(i_1,i_2,\ldots,i_n)$,  $i_1,i_2,\ldots,i_n\in \mathcal C_p$ 
such that $i_1\leq i_2\leq \cdots \leq i_n$, we have: 
\begin{equation} \delta^{(1)}_p(\mathbf i)=
\left\{\begin{array}{ll}
0, & \; {\rm if} \; \varepsilon_{(i_1,i_2,\ldots,i_n)}=1;\\
1+\frac{J^{(1)}_{p, \max}(\mathbf i)-J^{(1)}_{p, \min}(\mathbf i)}{2}, &
\; {\rm if} \; \varepsilon_{(i_1,i_2,\ldots,i_n)}=0,\\
\end{array}
\right.
\end{equation}
where
\begin{equation}
J^{(1)}_{p, \max}(\mathbf i)=\min_{1\leq \ell\leq \frac{n+1}{2}}\min\left(p-4,\sum_{t=1}^{n-2\ell}i_t-\sum_{s=n-2\ell +1}^ni_{s} +\ell(p-4)\right),
\end{equation}
\begin{equation}
J^{(1)}_{p, \min}(\mathbf i)=\max_{0\leq k\leq \frac{n+1}{2}}\max\left(0,\sum_{s=n-2k}^ni_{s} -\sum_{t=1}^{n-2k-1}i_t-k(p-4)\right).
\end{equation}
\end{lem}
\begin{proof}
Observe first from Lemma \ref{lemma:nulleven} that $\varepsilon_{(i_1,i_2,\ldots,i_n)}=0$ is a necessary condition for the existence of such $j$. 
Further, from the proof of Lemma \ref{lemma:deltaeven} we have 
to count the number of solutions $j\in \mathcal C_p$ 
of the double inequalities:  
\[  J^{(1)}_{p, \min}(\mathbf i)\leq 2j \leq J^{(1)}_{p, \max}(\mathbf i)\]
The claim follows by direct computation.
\end{proof}

\subsection{The order of Dehn twists in  the projective quantum representations}
The explicit description  from the previous section leads to the following computation:

\begin{lem}\label{lemma:orderodd}
Let $p\geq 5$ be odd, $c$  be an essential simple closed curve on $\Sigma_{g,n}$, $g\geq 1$ and 
  $i_1,i_2,\ldots,i_n$  be non-zero colors from $\mathcal C_p$. Then $\rho_{p,(i_1, \ldots, i_n)}(T_c)^p=1$. 
Furthermore,
\begin{enumerate}
\item 
Assume that one of the following conditions is satisfied:  
\begin{enumerate}
\item $c$ is separating of strictly positive genus; 
\item  $c$ is non-separating and either $g\geq 2$ or $g=1$ and $n=0$;
\item $c$ bounds a holed sphere with other boundary curves colored by the non-zero 
colors $i_1,i_2,\ldots,i_m$, for some $2\leq m\leq n$, where 
\begin{enumerate}
\item $\delta_p(i_1,i_2,\ldots,i_m)\geq 3$, or
\item $\delta_p(i_1,i_2,\ldots,i_m)=2$ and 
$J_{p,\min}(i_1,i_2,\ldots,i_m)=0$. 
\end{enumerate}
\item $c$ is non-separating, $g=1$ and 
\begin{enumerate}
\item $n\geq 1$, $\delta^{(1)}_p(i_1,i_2,\ldots,i_n)\geq 3$, or
\item $n\geq 1$, $\delta^{(1)}_p(i_1,i_2,\ldots,i_n)=2$ and 
$J_{p,\min}(i_1,i_2,\ldots,i_n)=0$. 
\end{enumerate}
\end{enumerate}
Then $\rho_{p,(i_1, \ldots, i_n)}(T_c)$  has order $p$.
\item If $c$ bounds a holed sphere with other boundary curves colored by the non-zero 
colors $i_1,i_2,\ldots,i_m$, for some $2\leq m\leq n$, where 
 $\delta_p(i_1,i_2,\ldots,i_m)=2$. Then  
 $\rho_{p,(i_1, \ldots, i_n)}(T_c)$  has order 
 \[p/{\rm g.c.d.}\left(J_{p,\max}(i_1,i_2,\ldots,i_m),p\right).\]
 \item $c$ is non-separating, $g=1$ and $n\geq 1$, $\delta^{(1)}_p(i_1,i_2,\ldots,i_n)=2$ and $J_{p,min} (i_1,i_2,\ldots,i_n) >0$.
 Then $\rho_{p,(i_1, \ldots, i_n)}(T_c)$  has order 
 \[p/{\rm g.c.d.}\left(\left\lfloor\frac{J^{(1)}_{p,\max}(i_1,i_2,\ldots,i_n)}{4}\right\rfloor,p\right).\]
 \item Suppose that: 
 \begin{enumerate}
 \item either $c$ bounds a holed sphere with other boundary curves colored by the non-zero 
colors $i_1,i_2,\ldots,i_m$, for some $2\leq m\leq n$, where 
 $\delta_p(i_1,i_2,\ldots,i_m)=1$, or 
\item $c$ is non-separating and $g=1$ $n\ge 1$ and  $\delta^{(1)}_p(i_1,i_2,\ldots,i_n)=1$.
\end{enumerate}
 Then $\rho_{p,(i_1, \ldots, i_n)}(T_c)$  has order $1$. 
 \end{enumerate}
\end{lem}
\begin{prv} 
Choose a pants decomposition of the surface $\Sigma_{g,n}$ obtained by deleting disjoint disks centered at the $n$ punctures, which contains the simple closed curve $c$. 
Note that boundary circles have to be part of the decomposition and are already colored as $(i_1,i_2,\ldots,i_r)$. 
If we cut  open $\Sigma_{g,n}$ along $c$ we either obtain $\Sigma_{g-1,n+2}$  -- when $c$ is non-separating -- or else a disjoint union 
$\sigma_{h,s}\sqcup \Sigma_{g-h,n+1-s}$.  We color the two new boundary circles of the resulting 
possibly disconnected surface by some color $j$ and let $W(\Sigma_{g,n}\setminus\{c\};j)$ denote the space of conformal blocks associated 
to it. Observe also that its identification with a subsurface of $\Sigma_{g,n}$ induces an injection of 
$W(\Sigma_{g,n}\setminus\{c\};j)$ into $W_{g,(i_1,i_2,\ldots,i_n)}$, and we have the following decomposition:  
\[ W_{g,(i_1,i_2,\ldots,i_n)}= \oplus_{j\in \mathcal C_p} W(\Sigma_{g,n}\setminus\{c\};j)\]
Now $\rho_{p, (i_1,i_2,\ldots,i_n)}(T_c)$ acts as a scalar on every subspace $W(\Sigma_{g,n}\setminus\{c\};j)$ and with respect to the direct sum decomposition above we have: 
\[ \rho_{p, (i_1, \ldots, i_n)}(T_c)=\oplus_{j\in \mathcal C_p} (-1)^jA^{j(j+2)}{\mathbf 1}_{W(\Sigma_{g,n}\setminus\{c\};j)}\]
where ${\mathbf 1}_{W}$ is the identity operator on the subspace $W$. 

On one hand eigenvalues of $\rho_{p, (i_1, \ldots, i_n)}(T_c)$ are powers of $A^8$ and hence 
\[ \rho_{p, (i_1, \ldots, i_n)}(T_c)^{p}=1\]

In order to compute the order of $\rho_{p,(i_1, \ldots, i_n)}(T_c)$ it suffices to find which colors $j$ appear {\em effectively} in the direct sum above, namely for which $j\in \mathcal C_p$ 
the vector space $W(\Sigma_{g,n}\setminus\{c\};j)$ is non-zero.

According to (\cite{LF}, Lemma 3.3) we have 
\[ \dim W_{1, p, (i,j)}= \frac{(p-1-\max(i,j))(\min(i,j)+1)}{2}\neq 0\]
Thus, if $c$ bounds subsurfaces of positive genera on both sides, then all possible colors could be realized effectively.   

If $c$ is non-separating and $g\geq 2$, the colors $j=0$ and $j=2$ do appear effectively 
and, in particular, $\rho_{p,(i_1, \ldots, i_n)}(T_c)$ has among its eigenvalues $1$ and $A^{8}$. 
Since the order of $A^8$ is $p$, the order of $\rho_{p,(i_1, \ldots, i_n)}(T_c)$ is actually $p$. 
This proves the first two items of our Lemma. 

If $c$ is like in (1.c.(i)), then there at least three colors $j$ which appear effectively. The proof of Lemma \ref{lemma:deltaodd} 
shows that these colors can be chosen to be consecutive even numbers $2k,2k+2,2k+4$. Since $\rho_{p,(i_1, \ldots, i_n)}(T_c)$
is the image of a diagonal matrix with $1,A^{8(k+1)}, A^{8(2k+3)}$ as coefficients,  $A^8$ is primitive $p-th$ root of $1$, 
and
${\rm g.c.d.}(p, k+1, 2k+3)=1$,  the order of $\rho_{p,(i_1, \ldots, i_n)}(T_c)$ is  $p$.   

If $c$ is like in (1.c.(ii)), then there are two colors which appear effectively, hence they are $0$ and $2$, which implies that the order $\rho_{p,(i_1, \ldots, i_n)}(T_c)$ is the order 
of $A^8$, namely $p$.   

When $g=1$ and $c$ is separating, then two copies of $c$ bound a holed sphere. Observe that $j,j,i_1,\ldots,i_m$   
satisfies the  conditions $P(k,m+2)$  for $j\in\{0,2\}$, if $i_1,\ldots,i_m$ satisfies the conditions  $P(k,m)$. Then Lemma \ref{lemma:null} permits to conclude the proof of the last item (1.d).

If $c$ bounds a holed sphere and $\delta_p(i_1,i_2,\ldots,i_m)=2$, then 
those $j$ which appear effectively are the consecutive even integers 
$J_{p,\min}(i_1,i_2,\ldots,i_m)$ and  $J_{p,\max}(i_1,i_2,\ldots,i_m)$. 
Then the order of $\rho_{p,(i_1, \ldots, i_n)}(T_c)$  is the order
of $A^{4 J_{p,\max}(i_1,i_2,\ldots,i_m)}$, which is 
 $p/{\rm g.c.d.}(J_{p,\max}(i_1,i_2,\ldots,i_m), p)$. 
 
 If $ c$ is non-separating, $g=1$ and  $\delta^{(1)}_p(i_1,i_2,\ldots,i_m)=2$, then 
the $j$ which appear effectively are the consecutive even integers 
$2\left\lceil\frac{J_{p,\min}(i_1,i_2,\ldots,i_m)}{4}\right\rceil$ and  $
2\left\lfloor\frac{J^{(1)}_{p,\max}(i_1,i_2,\ldots,i_m)}{4}\right\rfloor$. 
Then the order of $\rho_{p,(i_1, \ldots, i_n)}(T_c)$  is the order
of $A^{8 \left\lfloor\frac{J^{(1)}_{p,\max}(i_1,i_2,\ldots,i_m)}{4}\right\rfloor}$, which is 
$p/{\rm g.c.d.}\left(2\left\lfloor\frac{J^{(1)}_{p,\max}(i_1,i_2,\ldots,i_m)}{4}\right\rfloor,p\right)$. 

Eventually, if  $\delta_p(i_1,i_2,\ldots,i_m)=1$, or $\delta^{(1)}_p(i_1,i_2,\ldots,i_n)=1$ then there is only one color $j$ which could appear effectively as label of the curve $c$ and hence $\rho_{p,(i_1,i_2,\ldots,i_n)}(T_c)$ is a scalar, namely a trivial element in the projective unitary group. 
\end{prv}

\begin{lem}\label{lemma:ordereven}
Let $p\geq 6$ be even, $c$  an essential simple closed curve on $\Sigma_{g,n}$, $g\geq 1$ and 
$(i_1,i_2,\ldots,i_n)$ be non-zero colors from $\mathcal C_p$. We assume that 
\begin{equation}\label{parity} \sum_{s=1}^n i_s\equiv 0 \; ({\rm mod}\; 2),\end{equation}
as otherwise $W_{p,g,(i_1,i_2,\ldots,i_n)}$ is a null vector space. Then $\rho_{p,(i_1, \ldots, i_n)}(T_c)^{2p}=1.$ Furthermore: 
\begin{enumerate}
\item  If $c$ is separating of strictly positive genus, assume that the 
other boundary components of one connected complementary subsurface are 
colored $i_1,i_2,\ldots,i_m$, with $0\leq m\leq n$. 
\begin{enumerate}
\item If 
\begin{equation*}\label{paritysub3} \sum_{s=1}^m i_s\equiv 0 \; ({\rm mod}\; 2)
\end{equation*}
then $\rho_{p,(i_1, \ldots, i_n)}(T_c)$  has order
\begin{equation*}
\left\{\begin{array}{ll}
1, & {\rm if}\;p=6;\\
\frac p4,&  {\rm if }\; p\equiv 0 \; ({\rm mod} \; 4), p\geq 8;\\
\frac p2,&  {\rm if }\; p\equiv 2 \; ({\rm mod} \; 4),  p\geq 10.\\
\end{array}
\right.
\end{equation*}
\item If 
\begin{equation*}\label{paritysub1} \sum_{s=1}^m i_s\equiv 1 \; ({\rm mod}\; 2)\end{equation*}
then $\rho_{p,(i_1, \ldots, i_n)}(T_c)$  has order 
\begin{equation*}
\left\{\begin{array}{ll}
1, & {\rm if} \;  p\in\{6,8\};\\
5, & {\rm if }\; p=10;\\
2, & {\rm if} \;p=12;\\
\frac p2,& {\rm if} \; p\geq 14.\\
\end{array}
\right.
\end{equation*}
\end{enumerate} 
\item  If $c$ is non-separating and $g\geq 2$,  or $g=1$ and $n=0$, 
then $\rho_{p,(i_1, \ldots, i_n)}(T_c)$  has order 
\begin{equation*}
\left\{\begin{array}{ll}
4, & {\rm if} \;  p=6;\\
2p, & {\rm if }\; p\geq 8.\\
\end{array}
\right.
\end{equation*}
\item Suppose that $c$ bounds a holed sphere with other boundary curves colored by the non-zero 
colors $i_1,i_2,\ldots,i_m$, for some $2\leq m\leq n$,  where 
$\delta_p(i_1,i_2,\ldots,i_m)\geq 3$. 
\begin{enumerate}
\item If 
\begin{equation*}\label{paritysub} \sum_{s=1}^m i_s\equiv 0 \; ({\rm mod}\; 2),\end{equation*}
then $\rho_{p,(i_1, \ldots, i_n)}(T_c)$  has order 
\begin{equation*}
\left\{\begin{array}{ll}
1, & {\rm if} \;  p=6;\\
\frac p4, &  {\rm if }\; p\equiv 0 \; ({\rm mod} \; 4), p\geq 8\\
\frac p2,&  {\rm if }\; p\equiv 2 \; ({\rm mod} \; 4),  p\geq 8\\
\end{array}
\right.
\end{equation*}
\item If 
\begin{equation*}\label{paritysub2} \sum_{s=1}^m i_s\equiv 1 \; ({\rm mod}\; 2),\end{equation*}
then  $\rho_{p,(i_1, \ldots, i_n)}(T_c)$ has order
\[\left\{\begin{array}{ll}
1, & {\rm if}\; p\in\{6,8\};\\
5, & {\rm if}\; p=10;\\ 
2, & {\rm if}\; p=12;\\
\frac p2,& {\rm if}\; p\geq 14.\\
\end{array}
\right.
\]
\end{enumerate}

\item If $c$ is non-separating, $g=1$ and $\delta^{(1)}_p(i_1,i_2,\ldots,i_n)\geq 3$, so that $p\geq 8$,
 then $\rho_{p,(i_1, \ldots, i_n)}(T_c)$  has order  $2p$. 
\item Let $c$ bound a holed sphere with other boundary curves colored by the non-zero 
colors $i_1,i_2,\ldots,i_m$, for some $2\leq m\leq n$, where 
 $\delta_p(i_1,i_2,\ldots,i_m)=2$. Then  
 $\rho_{p,(i_1, \ldots, i_n)}(T_c)$  has order 
 \[p/2\,{\rm g.c.d.}\left(J_{p,\max}(i_1,i_2,\ldots,i_m),p/2\right).\]
 
 \item Let $c$ be non-separating, $g=1$,  $n\geq 1$ , $\delta^{(1)}_p(i_1,i_2,\ldots,i_n)=2$.
 Then $\rho_{p,(i_1, \ldots, i_n)}(T_c)$  has order 
 \[2p/{\rm g.c.d.}\left(2+J^{(1)}_{p,\max}(i_1,i_2,\ldots,i_n),p\right).\] 

\item Suppose that: 
 \begin{enumerate}
 \item either $c$ bounds a holed sphere with other boundary curves colored by the non-zero 
colors $i_1,i_2,\ldots,i_m$, for some $2\leq m\leq n$, where 
 $\delta_p(i_1,i_2,\ldots,i_m)=1$, or 
\item $c$ is non-separating and $g=1$, $n\ge 1$  and  $\delta^{(1)}_p(i_1,i_2,\ldots,i_n)=1$.
\end{enumerate}
 Then $\rho_{p,(i_1, \ldots, i_n)}(T_c)$  has order $1$. 
 \end{enumerate}
\end{lem}
\begin{prv} 
First note that the sum of the parity conditions over all vertices yields the global parity 
condition (\ref{parity}). Further, we follow the arguments from Lemma \ref{lemma:orderodd} above and use the same notation. 
We then have: 
\[ \widetilde{\rho}_{p, (i_1, \ldots, i_n)}(\widetilde{T_c})=\oplus_{j\in \mathcal C_p} (-1)^jA^{j(j+2)}{\mathbf 1}_{W(\Sigma_{g,n}\setminus\{c\};j)}\]
where ${\mathbf 1}_{W}$ is the identity operator on the subspace $W$. 
The eigenvalues of $\rho_{p, (i_1, \ldots, i_n)}(T_c)$ are powers of $A$ and hence 
\[ \rho_{p, (i_1, \ldots, i_n)}(T_c)^{2p}=1\]

In order to compute the order of $\rho_{p,(i_1, \ldots, i_n)}(T_c)$ it suffices to find which colors $j$ appear effectively in the direct sum above, namely for which $j\in \mathcal C_p$ 
the vector space $W(\Sigma_{g,n}\setminus\{c\};j)$ is non-zero. Let us first give some information about spaces of conformal blocks in positive genus: 

\begin{lem}

\begin{equation} \dim W_{1, p, (i,j)}= 
\left\{\begin{array}{ll}
\left(\frac{p-2}{2}-\max(i,j)\right)\left(1+\min(i,j)\right), &  {\rm if}\; i\equiv j \; ({\rm mod}\; 2);\\
0, & \; {\rm if}\; i\not\equiv j \; ({\rm mod}\; 2).
\end{array}\right.
\end{equation}
Hence $\dim W_{1, p, (i,j)}=0$ if only if $i\not\equiv j \; ({\rm mod}\; 2)$.
 
\end{lem}

\begin{prv}
 Direct computation. 
\end{prv}

\begin{coro}
 If $g\ge 2$, then $\dim W_{g, p, (i,j)}=0$ if only if $i\not\equiv j \; ({\rm mod}\; 2)$.
\end{coro}

\begin{prv}
Decompose a 2-holed surface of genus $g$ as $\Sigma_{g,2}= \Sigma_{1,2}\cup_{S^1} \Sigma_{1,2} \cup_{S^1} \ldots \cup_{S^1} \Sigma_{1,2}$ and apply the splitting principle
coloring each glued boundary circle with the color $i$. 
\end{prv}

\begin{coro}\label{coro:bcgnonva}
 If $g\ge 1$, $m\ge 2$ $\dim W_{g, p, (i_1,\ldots i_m)}\not =0$ if only if $$\sum_{t=1}^m i_t \equiv 0 \; ({\rm mod}\; 2).$$
\end{coro}

\begin{prv} The case $m=2$ follows from the previous statements. So we may assume  $m \ge 3$. 
Decompose a 2-holed surface of genus $g$ as $\Sigma_{g,m}= \Sigma_{1,2}\cup_{S^1} \Sigma_{1,2} \cup_{S^1} \ldots \cup_{S^1} \Sigma_{1,2}\cup_{S^1} \Sigma_{0,m}$
assuming $i_1$ colors the left boundary curve of the first $\Sigma_{1,2}$ and apply the splitting principle
coloring each glued boundary circle with a color $j_1$ such that $W_{0,p, (j_1, i_2, \ldots, i_m) } \not=0$ which exists by Lemma \ref{lemma:deltaeven} and has the same parity as $i_1$. 
\end{prv}

 Suppose now that $c$ bounds two subsurfaces of positive genera on both sides. As above,  the colors $j$ 
 that could be realized effectively on $c$ are precisely those satisfying the parity condition:
\[ j\equiv \sum_{s=1}^m i_s \; ({\rm mod}\; 2)\]  
and we call the splitting of the surface odd/even according to the parity of the sum of colors on either side. 

Let $c$ provide an even splitting. If $p=6$ there is only one even color, namely $0$ and hence 
\[ \rho_{g,6, (i_1,i_2,\ldots,i_n)}(T_c)=1\]
For even $p\geq 8$, there exist at least two even colors $0$ and $2$ which appear effectively thanks to Corollary \ref{coro:bcgnonva}. 
All other even labels yield eigenvalues which are powers of $A^8$. 
Since $A^8$ has order $\frac p4$, when $p\equiv 0 \; ({\rm mod}\; 4)$ and $\frac p 2$, otherwise, 
the same holds for $\rho_{g,p, (i_1,i_2,\ldots,i_n)}(T_c)$. 

Let $c$ provide an odd splitting. If $p=6$ or $p=8$ there is only one odd color in $\mathcal C_p$, 
namely $1$. The matrix so obtained is then scalar and hence its image in the projective unitary group is trivial:
\[ \rho_{g,6, (i_1,i_2,\ldots,i_n)}(T_c)=1, \; \rho_{g,8, (i_1,i_2,\ldots,i_n)}(T_c)=1\]
If $p=10$ or $p=12$,  there are exactly two odd colors $1$ and $3$. They effectively appear by Corollary \ref{coro:bcgnonva}.  The associated 
eigenvalues are $-A^3$ and $-A^{15}$. The smallest exponent for which 
the powers of these two eigenvalues coincide are $2$, for $p=12$ and $5$ for $p=10$. 
For $p\geq 14$ , there are at least three odd colors $1,3$ and $5$, with associated eigenvalues 
$-A^3,-A^{15}$ and $-A^{35}$. The smallest $k$ for which $k$-th powers are equal should verify 
$A^{12k}=A^{20k}=1$ and hence $A^{4k}=1$, so that $k=\frac p 2$. 
Higher odd colors lead to eigenvalues of the form powers of $A^4$ times $-A^3$. 
Thus the order of $\rho_{g,p, (i_1,i_2,\ldots,i_n)}(T_c)$ is $\frac p 2$. 
This proves the first item of our Lemma.

Assume that $c$ is nonseparating and $g\geq 2$. Choose a pants decomposition such that 
$c$ is a meridian of a one holed torus bounded by the simple closed curve $d$. 
The parity obstruction for the complementary subsurface shows that $d$ must be even in any 
$p$-admissible coloring.  By Corollary \ref{coro:bcgnonva}, any even color for $d$ 
appears effectively. Therefore any color for $c$ appears effectively, in particular 
the colors $j=0, 1$ and also $j=2$, when $p\geq 8$. 
Now $-A^3$ had order $4$, for $p=6$.  If $p\geq 8$, the l.c.m. of orders of 
$-A^3$ and $A^8$ is $2p$ and hence $\rho_{p,(i_1, \ldots, i_n)}(T_c)$ has order $2p$. 
The proof is similar for $g=1$, $n=0$. This proves the second  item of the Lemma.

Assume  $c$ is separating a holed sphere, providing an even splitting. 
If $p=6$, there is only one even color $0$, and hence $\rho_{6,(i_1,i_2,\ldots,i_n)}=1$. 
If $\delta_p(i_1,i_2,\ldots,i_m)\geq 3$, 
there are at least three consecutive even colors $j$ which appear effectively for $c$. 
Then the order of  $\rho_{p,(i_1,i_2,\ldots,i_n)}$ is the order of $A^8$, which is as stated, by the same argument as in the odd case. 
If $\delta_p(i_1,i_2,\ldots,i_m)=2$ and $J_{p,\min}(i_1,i_2,\ldots,i_m)=0$, then the 
colors $0$ and $2$ appear effectively and hence the order of $\rho_{p,(i_1,i_2,\ldots,i_n)}$ is the order of $A^8$ again. 

Assume  $c$ is separating a holed sphere, providing an odd splitting. 
If $p=6$ or $p=8$, there is only one odd color $1$, and hence $\rho_{p,(i_1,i_2,\ldots,i_n)}=1$. 
If $\delta_p(i_1,i_2,\ldots,i_m)\geq 3$, 
there are at least three consecutive odd colors $j$ which appear effectively for $c$ and the order 
$\rho_{p,(i_1,i_2,\ldots,i_n)}$ is the order of $A^4$, which is as stated, except for small values of $p\leq 12$. This concludes the proof of the third item.

When $g=1$ and $c$ is non-separating and $\delta^{(1)}_p(i_1,i_2,\ldots,i_m)\geq 3$
then there are at least three consecutive colors $j,j+1,j+2$ which appear effectively. 
The corresponding eigenvalues have a common $N$-th power iff 
$(-A)^{(2j+3)N}=(-A)^{(2j+5)N}=1$, and hence $A^{N}=1$, so that $N$ is divisible by $2p$. 
This proves the fourth item.

If $c$ bounds a holed sphere and  $\delta_p(i_1,i_2,\ldots,i_m)=2$, then $j$ must be even and 
there are exactly two even colors $j$ and $j+2$ which appear effectively and the 
order of $\rho_{p,(i_1,i_2,\ldots,i_n)}$ is the order of $A^{4(j+2)}$, which is as claimed. 

If $c$ is non-separating, $g=1$, $\delta^{(1)}_p(i_1,i_2,\ldots,i_m)=2$, then there are exactly two 
 colors $j$ and $j+1$ which appear effectively and the
order of $\rho_{p,(i_1,i_2,\ldots,i_n)}$ is the order of $A^{2j+3}$ which is  as claimed. 

Eventually, recall that if  $\delta_p(i_1,i_2,\ldots,i_m)=1$ or 
$\delta^{(1)}_p(i_1,i_2,\ldots,i_m)=1$, then $\rho_{p,(i_1,i_2,\ldots,i_n)}$ is a scalar and hence 
of order $1$. 
\end{prv}

\begin{rem}
Assume now that all $i_s$ are equal to $\frac{p-4}{2}$.
We claim that the space $W_{0,p,(i_1,\ldots,i_n)}$ could only be non-trivial when 
$n\geq 4$ is even, in which case it is of dimension $1$. This follows either by using 
Lemma \ref{lemma:nulleven}, or else directly by analyzing the dual tree of the pants decomposition 
used in its proof. The internal edges form a chain; their labels are forced by the 
$p$-admissibility condition to be alternatively $0,\frac{p-4}{2},0,\frac{p-4}{2},\ldots$.
This implies that the only color which could appear effectively on $c$ is $0$, if $m$ is odd and 
$\frac{p-4}{2}$, if $m$ is even. In both cases  $\rho_{p,(i_1,i_2,\ldots,i_n)}$ is a scalar and hence it is of order $1$. 
\end{rem}

\subsection{Infiniteness results}
Recall from the introduction that $\Mod(\Sigma_g^n[\mathbf k])$ denotes the normal subgroup generated 
by $k_0$-th powers of Dehn twist along non-separating simple closed curves and 
$k_j$-th powers of Dehn twists along  simple closed curves of type $j$, where $\mathbf k=(k_0;k_1,k_2,\ldots,k_{N_{g,n}})$.
As a shortcut we use $\Mod(\Sigma_g^n[k;m])$ for $\mathbf k=(k,m,m,m\ldots,m)$ and 
 $\Mod(\Sigma_g^n[k;-])$ for $\mathbf k=(k;)$, where $k_i$ are absent for $i>0$.

\begin{prop}\label{factori} Assume $g\ge 1$, $n\ge 0$. Let $(i_1,\ldots,i_n)\in \mathcal{C}_p^n$   such that 
\[ \sum_{t=1}^n i_t \equiv 0 \; ({\rm mod }\; 2).\]
\begin{enumerate}
 \item For odd $p\geq 5$,  $\rho_{p,(i_1,i_2,\ldots,i_n)}$ factors through $\Mod(\Sigma_g^n)/\Mod(\Sigma_g^n)[p]$. 
\item 
For even $p\geq 10$,  $\rho_{p,(i_1,i_2,\ldots,i_n)}$ factors through 
$\Mod(\Sigma_g^n)/\Mod(\Sigma_g^n)[2p; p/{\rm g.c.d.}(4,p)]$, if all colors are even, 
and through out  $\Mod(\Sigma_g^n)/\Mod(\Sigma_g^n)[2p; p/2]$, in general.
\end{enumerate}
\end{prop}
\begin{proof}
This is immediate from Lemmas \ref{lemma:orderodd} and \ref{lemma:ordereven}. 
\end{proof}

\begin{prop}\label{infty-closed}
Assume that $g\ge 2$ and  $p\ge 5$, if $p$ is odd, or $p\ge 10$  
and $(g,p)\neq (2,12)$, if $p$ is even. 
Then  $\rho_p(\Mod(\Sigma_{g,n}))$ is infinite. 
 \end{prop}

\begin{proof}
For all $p\geq 5$ except $p\in\{6,8,12\}$, $(g,p)=(2,20)$ 
this is proved in \cite{F}; Masbaum found explicit elements of infinite order in \cite{Mas}. 
The case $p=20$ is settled in (\cite{EJ}, Appendix A). 

It remains to prove the claim for $g\geq 3$ and $p=12$. 
Note that Wright proved that the image is finite when $p=12$ and $g=2$ (see \cite{Wr2}). 
We follow the approach developed by Coxeter (see \cite{Cox1,Cox2}) who used it to prove 
that quotients of braid groups by powers of braid generators is infinite. This was used 
in \cite{Kor} to study the case of punctured  torus. 

\begin{lem}[\cite{Cox1},p.116, or \cite{Cox2},p.121]\label{Coxeter}
If a group acts irreducibly on a finite dimensional complex vector space 
keeping invariant a non-degenerate indefinite  Hermitian form, then
this group must be infinite. 
\end{lem}
  
We need also the following:

\begin{lem}[\cite{KoSa2}]\label{irred}
Let $p$ be even. Then the action of $\widetilde{\rho}_{p,(1,i_2,\ldots,i_n)}(\widetilde{\Mod(\Sigma_{g,n})})$ on
the space $W_{g,p,(1,i_2,\ldots,i_n)}$ is irreducible for all $g\geq 1$ and all colors 
$(1,i_2,\ldots,i_n)$ for which this vector space is non-zero.  
\end{lem}
  
Consider now a torus $\Sigma_{1,2}$ with two boundary components colored $(1,1)$. 
We need the following result whose proof will be postponed a few lines later:
\begin{lem}\label{12}
The Hermitian form on the space $W_{1,p, (1,1)}$ is indefinite, if $p\geq 10$ is even or $p\geq 5$ is odd. 
\end{lem}  
  
The three lemmas above imply that the image of $\rho_{12,(1,1)}(\Mod(\Sigma_{1,2}))$ 
is infinite. Further observe that for every $g\geq 3$  the embedding $\Sigma_{1,2}\subset \Sigma_{g}$ 
obtained by gluing $\Sigma_{g-2,2}$ to  $\Sigma_{1,2}$ along the boundary components induces an 
injection at the mapping class group level: $\Mod(\Sigma_{1,2})\to \Mod(\Sigma_g)$. 
It follows that $\rho_{12}(\Mod(\Sigma_g))$ contains the subgroup $\rho_{12,(1,1)}(\Mod(\Sigma_{1,2}))$ 
and it is therefore infinite. This proves Proposition \ref{infty-closed}. 
\end{proof}

\begin{proof}[Proof of Lemma \ref{12}]
The graph $G$ consists of a loop with two pending edges, labeled $1$. Let $p$ be even. 
The two remaining edges, whose union form the loop, are colored 
by $(i,j)$; then by the admissibility conditions at level $p$ we 
need that $i+j$ be odd, smaller than or equal to $p-5$ and such that 
$(1,i, j)$ satisfy the triangle inequalities, so that $|i-j|=1$. 
Thus the admissible colorings are of the form $(k,k+1)$ and $(k+1,k)$, with $k=\left\{0,1,2,\frac{p-4}{2}\right\}$, so that 
$\dim W_{1,p, (1,1)}=p-4$. Moreover, by the formulas above the diagonal term of the Hermitian form 
corresponding to the coloring $(i,j)$ is $\eta^2\frac{1}{[i]![j]!}$. 
The quotient of two diagonal terms associated to the colorings $(k,k+1)$ and $(k+1,k+2)$
is therefore given by:
\[ [k+1][k+2].\]
Since $A$ was a primitive $2p$-th root of unity, 
$A=\exp\left(\frac{i\pi \ell}{p}\right)$, with ${\rm g.c.d.}(\ell, 2p)=1$. Now 
$[k]=\frac{\sin\left(\frac{2i\pi \ell k}{p}\right)}{\sin\left(\frac{2i\pi \ell}{p}\right)}$. 
If $p\geq 10$, there exists always some $\ell$ such that $[1][2] <0$, namely 
the two diagonal terms have opposite sign and 
thus the Hermitian form is indefinite. The case $p$ odd is similar and we skip the details. 
\end{proof}

The remaining cases, not covered by the Proposition \ref{infty-closed} will be discussed in the next section.

Choose a basepoint on the first boundary circle of  $\Sigma_{g,n}$. 
If we adjoin a disk capping off the first hole the embedding 
$\Sigma_{g,n} \to \Sigma_{g,n-1}$ induces an exact sequence:  
\[ 1\to\pi_1(U\Sigma_{g,n-1})\to  \Mod(\Sigma_{g,n})\to \Mod(\Sigma_{g,n-1})\to 1.\]
Here $\pi_1(U\Sigma_{g,n-1})$ is the fundamental group of the unit tangent bundle 
at $\Sigma_{g,n-1}$, namely a central extension of $\pi_1(\Sigma_{g,n-1})$. 
Then the projective representation $\rho_{p, (i_1,i_2,\ldots,i_n)}$  of $\Mod(\Sigma_{g,n})$ 
induces a representation for which we keep the same notation: 
\[  \rho_{p, (i_1,i_2,\ldots,i_n)}: \pi_1(\Sigma_{g,n-1})\to PU(W_{g,p,(i_1,i_2,\ldots,i_n)}).\]

Recall that a subsurface $\Sigma_{g',n'}$ of $\Sigma_{g,n}$ is {\em essential} if  
the morphism $\pi_1(\Sigma_{g',n'})\to \pi_1(\Sigma_{g,n})$ induced by the inclusion is injective. 
It is well-known that a subsurface is essential  if its complement does not contain 
disk components. 

\begin{prop}\label{inffib} Assume $n\ge 1$. 
Let $\Sigma\cong\Sigma_{0,3}$ or $\Sigma_{1,1}$ be an essential subsurface of $\Sigma_{g}^{n-1}$ equipped with an orientation preserving homeomorphism onto a general fiber of the  map $f^1_{n}: {{\mathcal M}_g^{n}}
 \to{{\mathcal M}_g^{n-1}}$ which forgets the first puncture. 
\begin{enumerate}
\item Assume that $g\geq 2$ and $p\geq 5$ is odd. Then 
$\rho_{p,(i_1,i_2,\ldots,i_n)}(\pi_1(\Sigma))$ contains a free nonabelian group, if $i_1\neq 0$. 
\item Assume that $g\geq 2$ and $p\geq 10$ is even, $(g,p)\neq (2,12)$,   
$n\geq 2$ and $1+i_2+\cdots+i_n\equiv 0\; ({\rm mod}\; 2)$. 
Then $\rho_{p,(1,i_2,\ldots,i_n)}(\pi_1(\Sigma))$ contains a free nonabelian group. 
\item Assume that $g\geq 2$ and $p\geq 10$ is even, $(g,p)\neq (2,12)$,   
$n\geq 1$ and $i_2+\cdots+i_n\equiv 0\; ({\rm mod}\; 2)$. 
Then $\rho_{p,(2,i_2,\ldots,i_n)}(\pi_1(\Sigma))$ contains a free nonabelian group.
\end{enumerate}
\end{prop}
\begin{proof}
When $i_1=2$ this is the main result of (\cite{KoSa}, Thm.4.1) noting that their proof works for all $p\geq 5$ not only for large enough $p$. 
For other values of $i_1\neq 0$ this is contained in the proof of (\cite{LF}, Prop. 3.2, see also \cite{F2}).

Let $c$ be a curve separating the subsurface $\Sigma_{1,2}$ of $\Sigma_{g,n}$ with boundary labeled $1$.  
As $g\geq 2$, the color $j=1$ appears effectively for the curve $c$. 
Then Lemma \ref{12} implies that the Hermitian form on the subspace 
$W_{1,p, (1,1)}$ is indefinite, and hence the Hermitian form on 
$W_{g,p,(1,i_2,\ldots,i_n)}$ is indefinite (when the space is nonzero). 
Then the claim follows from Lemmas \ref{Coxeter} and 
\ref{irred}. 

The last item follows by observing that the proof given in \cite{KoSa} for $i_1=2$, can be adapted without any 
change to even $p$ (see \cite{LF}).  
\end{proof}

\subsection{Finiteness for small values of $p$}
We start by recalling the following well-known results of Humphreis (\cite{Hum}):
\begin{lem}[\cite{Hum}]
\begin{enumerate}
\item 
The group $\Mod(\Sigma_{g,n})/\Mod(\Sigma_{g,n})[2; -]$ is finite for every $g\geq 1$, $n\leq 1$, 
and, in particular,   $\Mod(\Sigma_{g,n})/\Mod(\Sigma_{g,n})[2]$ is finite. Moreover, 
\[\Mod(\Sigma_{g,n})[2; -]=\Mod(\Sigma_{g,n})[2; 1]=\ker(\Mod(\Sigma_{g,n})\to Sp(2g,\Z/2\Z)).\]
\item The group $\Mod(\Sigma_{g,n})/\Mod(\Sigma_{g,n})[3; -]$ is finite for $(g,n)\in\{(2,0), (2,1),(3,0)\}$, 
and, in particular, $\Mod(\Sigma_{g,n})/\Mod(\Sigma_{g,n})[3]$ is finite. Moreover, 
\[\Mod(\Sigma_{2})[3; -]=\Mod(\Sigma_{2})[3; 1]\supset\ker(\Mod(\Sigma_{2})\to Sp(2g,\Z)),\]
\[\Mod(\Sigma_{3})[3; -]=\Mod(\Sigma_{3})[3; 1].\]
\item For $n\geq 0$, $p\geq 4$, the group $\Mod(\Sigma_{2,n})/\Mod(\Sigma_{2,n})[p; -]$ is infinite. 
\end{enumerate}
\end{lem}

Further $\rho_6$ factors through $\Mod(\Sigma_g)/\Mod(\Sigma_g)[4;1]$, which is finite.

Wright proved in \cite{Wr1} and \cite{Wr2} the following finiteness results:
\begin{lem}[\cite{Wr1}]
The image $\rho_8(\Mod(\Sigma_g))$ is finite for every $g\geq 1$. 
\end{lem}

\begin{lem}[\cite{Wr2}]
The image $\rho_{12}(\Mod(\Sigma_2))$ is finite for  $g=2$. 
\end{lem}
 Note that  $\rho_{12}(\Mod(\Sigma_g))$ is a quotient of $\Mod(\Sigma_g)/\Mod(\Sigma_g)[24;3]$. 
 
 This has to be compared with 
 the finiteness question about  $\Mod(\Sigma_g)/\Mod(\Sigma_g)[3]$, which was also answered 
 affirmatively only for $g\in\{2,3\}$, by Humphries (see \cite{Hum}).

\section{TQFT representations on isotropy groups of  stable curves}

\subsection{Stabilizers of pants decompositions}
The aim of this section is to describe the image of the stabilizer of a pants decomposition by quantum representations. If $P$ is a pants decomposition  of some surface $\Sigma_{g,n}$ we keep the same notation for the isotopy class of the multicurve consisting of all loops from $P$. Note that the order of the 
curves is irrelevant. Set further $\Mod(\Sigma_{g,n}, P)$ for the pure mapping classes of orientation-preserving homeomorphisms which 
preserve  the isotopy class of the  multicurve $P$ and preserve pointwise 
the boundary components. 
 
Let $G$ be a uni-trivalent  graph embedded in the handlebody $H_g$ in such a way that 
the endpoints of $G$ sit on the boundary and $H_g$ retracts onto $G$. 
Let $\Sigma_{g,n}$ be the result of drilling  small holes around the endpoints of $G$ within the boundary surface. 
Then $\Sigma_{g,n}$ inherits  a pants decomposition $P=P(G)$ which consists of the set of 
simple loops $\gamma_e$, where $\gamma_e$ bounds a small disk intersecting once $e$, for every 
edge $e$ of $G$. 

Let $\Aut(G)$ denote the group of  automorphisms of $G$ which preserve pointwise the leaves.

The graph $G$ has two types of internal trivalent vertices, as follows. 
A {\em  generic} vertex has 3 distinct incoming edges and provides an essential $\Sigma_{0,3} \subset \Sigma_{g,n}$
and the {\em tadpole} has only 2 distinct incoming edges and gives raise to an essential $\Sigma_{1,1} \subset \Sigma_{g,n}$. 
Fix an orientation of the tadpole loops and define $\Aut^+(G)\subset \Aut(G)$ to be the subgroup preserving the orientation of the tadpoles.  

If $P\subset \Sigma_{g,n}$ is a pants decomposition of $\Sigma_{g,n}$ we denote by $T(P)$ the free abelian group generated by the Dehn twists along the non-peripheral curves in $P$ and 
the $n$ boundary components.  Thus $T(P)$ is isomorphic to $\Z^{3g-3+n}\times \Z^n$.

\begin{lem}\label{lemma:stabilizer}
The group   $\Mod(\Sigma_{g,n}, P)$ is an extension 
\begin{equation}\label{exactaut}
1\to T(P)\to \Mod(\Sigma_{g,n},P)\to \Aut(G)\to 1. 
\end{equation}
\end{lem}
\begin{proof}
Any mapping class preserving $P$ induces an automorphism of $G$. On tadpole loops, the orientation is preserved if the mapping class preserves the
two connected components of a deleted annulus neighborhood of the simple curve in the one holed torus corresponding to the loop and reversed if these two components are exchanged. 
If this automorphism is trivial 
then each pair of pants in the decomposition $P$ is fixed. It is well-known that the 
mapping class group $\Mod(\Sigma_{0,3})$ of a pair of pants is the abelian group generated by 
the boundary Dehn twists.

This establishes the exact sequence above, except for the surjectivity. First of all, observe 
$\Aut(G)=(\Z/2\Z)^t \rtimes \Aut^+(G)$ where $t$ is the number of tadpoles. There is a mapping class 
 $h\in \Mod(\Sigma_{1,1})$ which lifts the elliptic involution $-\mathbf 1 \in SL_2(\Z)=\Mod(\Sigma_1^1)$. Its square is the Dehn twist along the boundary curve. 
Given $(\epsilon_i)_{1\le i \le t} \in (\Z/2\Z)^t$, we can glue $h^{\epsilon_k}$ on the $k$-th tadpole and the identity on the rest of the surface to get a mapping class that projects to 
$(\epsilon_i)_{1\le i \le t}$ in the first factor. 

If $\phi \in \Aut^+(G)$ and $\Pi$ is the pair of pants corresponding to a generic vertex $v$,  we construct a homeomorphism from $\Pi$ to $\Pi'$ the pair of pants
corresponding to $\phi(v)$ mapping in an orientation preserving fashion the three boundary curves  represented by the incident edges $\{e, e', e''\}$ to the 
three boundary curves represented by $\{\phi(e), \phi(e'), \phi(e'')\}$. These homeomorphisms glue into a homeomorphism of $\Sigma_{g-t, n+t}= \Sigma_{g,n}\setminus \cup_{1\le k \le t} \Sigma_{1,1}$ preserving $P\cap \Sigma_{g-t, n+t}$. Then, we can  extend $h$ to a
homeomorphism $\bar h$ of $\Sigma_{g,n}$ preserving $P$ and the orientation on tadpoles. Obviously $\bar h$ maps to $\phi$. 

\end{proof}

The following lemma will not be used  here, but we include it nevertheless. 

\begin{lem} The exact sequence (\ref{exactaut}) splits over $\Aut^+(G)$. 
 \end{lem}
\begin{proof}

Each pair of pants decomposes into the union of two hexagons glued along 
three disjoint segments in their boundary.  On each pair of pants we color one hexagon in black, the other one in white. This expresses $\Sigma_{g,n}$ as an union 
of $6g-6+2n$ hexagons, $3g-3+n$ of them being black and $3g-3+n$ being white. 

At each boundary component of a pant we may adjust locally the hexagons so that the parts of their boundaries that are internal to the pant match and glue into a 
union of simple closed curves  and that the black hexagons of two adjacent pants (including self-adjacent ones) match at the boundary of the pant. 

This induces a partition of $\Sigma_{g,n}$ into two (non necessarily connected) subsurfaces, one being white, the other one black. We consider the group $A$
of mapping classes of orientation preserving homeomorphisms of $\Sigma_{g,n}$ that permute the maximal set of $3g-3+n$ non-peripheral simple curves and preserve the 
black/white partition. 

There is natural morphism $A \to \Aut^+ (G)$, where $ \Aut^+ (G)$ fixes an orientation of the tadpole loops.  Now an element of $\Aut^+ (G)$
defines a permutation of the set of pairs of black and white hexagons in the same pant respecting boundary hence lifts to an element of $A$.
The kernel of this morphism is trivial since the mapping class group of an hexagon (or a disk) modulo its boundary is trivial, by Alexander's lemma. 
\end{proof}

Let $\mathbf k$ be a integral vector indexed by the set  of $\Aut(G)$-orbits of 
non-peripheral simple closed curves and a component $1$ attached to  the boundary components on $(\Sigma_{g,n}, P)$. 
We denote by $T(P)[\mathbf k]$ the free abelian subgroup  of $T(P)$ 
generated  by $k_j$-th powers of Dehn twists along curves in $P$ of type $j$. 
Then $T(P)[\mathbf k]$ is a normal subgroup of $\Mod(\Sigma_{g,n},P)$ and
 we have  an extension
\begin{equation}
1\to \frac{T(P)}{T(P)[\mathbf k]}\to \frac{\Mod(\Sigma_{g,n},P)}{T(P)[\mathbf k]}
\to \Aut(G)\to 1.
\end{equation}

We use the shorthand notation $[q,r]$ to denote the integral vector whose components are $q$ on the $\Aut(G))$-orbits of non-separating curves and $r$ on the orbits of the separating ones.

Our main result here is then the following: 

\begin{prop}\label{prop:stabilizer}
Let $\mathbf i=(2,2,\ldots,2)$, $g>0$, $n\geq 0$ and 
$2g-2+n>0$ and $(g,n)\not=(1,0),(1,1),(2,0)$.
\begin{enumerate}
\item If $p\geq 5$ is odd, then $\rho_{p,(\mathbf i)}(\Mod(\Sigma_{g,n},P))$ is isomorphic to 
$\frac{\Mod(\Sigma_{g,n},P)}{T(P)[p]}$. 
\item If $p\geq 8$ is even then $\rho_{p,(\mathbf i)}(\Mod(\Sigma_{g,n},P))$ is isomorphic to 
$\frac{\Mod(\Sigma_{g,n},P)}{T(P)[2p, p/{\rm g.c.d.}(p,4)]}$. 
\end{enumerate}
\end{prop}
\begin{proof}
Let us first investigate   $\ker\rho_{p,(\mathbf i)}\cap T(P)$. The description for $\rho_p$ shows that 
$\rho_p|_{T(P)}$ maps into a maximal torus of the projective group. In the standard basis attached to $P$,  all matrices associated to the Dehn twists 
along curves in $P$ are simultaneously diagonal. 

We did not look at the order of Dehn twists in genus $0$. However, for this particular vector of colors $\mathbf i=(2, \ldots, 2)$, in any genus, 
we can always color the internal edges except one with the color $2$ and
the remaining one with the colors $0$ or $2$. Hence we are reduced  to the same
calculation as in the proof of Lemma \ref{lemma:orderodd} .(1).(c) resp. Lemma \ref{lemma:ordereven} .(3).(a) for
separating curves.  So the order is always $p$ if $p$ is odd and is $p/{\rm g.c.d.}(p,4)$ if $p$ even.  In higher genus, nothing changes 
for separating curves. Since
the  non-separating edges can be colored with the colors 
$0,2$ if $p$ is odd, with the colors $0,1,2$ if $p$ is even, the non-separating order is $p$ when $p$ is odd,  $2p$ when $p$ is even. 

It follows that  
$T(P)[p] \subseteq \ker\rho_{p,(\mathbf i)}\cap T(P)$, if $p$ is odd and $T(P)[2p, p/{\rm g.c.d.}(p,4)] \subseteq \ker\rho_{p,(\mathbf i)}\cap T(P)$, if $p$ is even, respectively.

We prove by induction the genus that these inclusions are equalities. 

First if $g=0$ the case $n=3$ is trivial, the cases $n=4,5$ follow by inspection. Assume now that the induction hypothesis holds 
for all $k\le n$ and consider $(\Sigma_0,_{n+1}, P)$ a pants decomposition. 
Let us fix an enumeration $(e_i)_{1\le i \le 3g-3+n+1}$ of the internal edges of $G$ which is a tree. 
Suppose that we have a relation of the form 
\[ \prod_{i=1}^{3g-3+n+1} \rho_{p, (\mathbf{i})}(T_{c_i})^{m_i}=1.
\]
We view this as an identity in $PGL(W(\Sigma_{0,n+1})_{(\mathbf{i})})$. Let us cut along the internal edge $e_1$ and get two holed spheres with $m$, resp. $m'$ number of holes
 $3\le m,m'\le n$. Color the respective boundary components corresponding to $e_1$ with the color $2$. This gives 
an inclusion of vector spaces 
\[ W((\Sigma_{0,m})_{(\mathbf{i})})\oplus W((\Sigma_{0,m'})_{(\mathbf{i})})\subset W((\Sigma_{0,n+1})_{(\mathbf{i})}).
\] 
Moreover $\widetilde{\rho}_{p,(\mathbf{i})}$ preserves the two factors and restrict to their $\widetilde{\rho}_{p,(\mathbf{i})}$
on their $T(P)$. It follows that the separating order divides $e_i$ for $i\ge 2$. Hence the relation reduces to $\rho_{p, (\mathbf{i})}(T_{c_1})^{m_1}=1.$
And the separating order divides $m_1$ as desired. 

The case $g=0$ being settled, the higher genus case follows by induction on $g$. One uses the same argument as above cutting along non-separating edges. 
One gets only a subspace 
$$W((\Sigma_{g,n})_{(\mathbf{i})})\subset W((\Sigma_{g+1,n})_{(\mathbf{i})}).
$$
 the important point is the functoriality of the restriction of $\widetilde{\rho}_{p,(\mathbf{i})}$ to this subspace. 
 
 The kernel of $\rho_{p,(\mathbf i)}:{ \frac{\Mod(\Sigma_{g,n},P)}{T(P)[p]}}\longrightarrow PGL(W((\Sigma_{g,n})_{(\mathbf{i})})$ is thus a subgroup $K$ which maps
 isomorphically onto its image in $\bar K < \Aut(G)$. 
 The action of $\Aut(G)$ on the admissible colorings of $G$ is then an invariant subset of its action by $\rho_{p,(\mathbf{i})}$ on the projective space of conformal
 blocks  $\mathbb{P} (W((\Sigma_{g,n})_{(\mathbf{i})}))$. It follows that $\bar K$ is contained in the subgroup of $\Aut(G)$ which fixes all admissible colorings of $G$. 
 Since one can always put the color 2 on all edges except one, it follows that $\bar K$ fixes all edges of $G$. 
 
 Assume $k\in \bar K$ does not fix a generic vertex $v$. Then the three edges emanating from $v$ arrive at $k.v$. It follows that $G$ is a theta graph. Hence $(g,n)=(2,0)$.  
Assume $k$ fixes every generic vertex and does not fix a tadpole vertex $v$. It follows that $k$ permutes 2 tadpole vertices joined by an edge and $(g,n)=(2,0)$ too. 
Actually we reach a contradiction in this case since it permutes the two tadpole loops. 

In all other cases $k$ fixes all vertices and edges. Hence, $\bar K$ is a subgroup of the $(\mathbb{Z}/2\mathbb{Z})^t$ corresponding to the tadpoles hence to essential subsurfaces $\Sigma_{1,1}\subset \Sigma_{g,n}$.
If we take the $h \in Mod(\Sigma_{1,1})$ be the lift of  the elliptic involution and push it as a mapping class of $\Sigma_{g,n}$ we have $h^2=T_c$ where $c$ is the boundary of $\Sigma_1^1$. Hence
if this boundary is non peripheral we deduce from  $\rho_{p,1,(2)}(T_c)\not =1$ that $\rho_{p,1,(2)}(h)\not =1$. We conclude as above by induction on $g$. The excluded cases are exceptions, 
the corresponding elements being central in the corresponding mapping class groups.  
\end{proof}

\subsection{Isotropy groups  of general stable curves} \label{subsect:isotropy}

If $C$ is a stable curve of genus $g$ with $n$ marked punctures, its automorphism group acts naturally on the dual graph $\Gamma_C$  of $C$, 
the unoriented marked graph whose vertices are the irreducible components of $C$ marked
by their genus and the marking of the punctures they contain and whose  edges are the double points connecting them, loops being allowed. 
A topological model of $C$ is obtained from a closed oriented surface $S$ of genus $g$ with $n$ marked points by pinching 
every loop of the multicurve $\underline\alpha$ to a point. Then
$\Gamma_C$ is isomorphic to the weighted dual graph $\Gamma(\underline\alpha)$  of 
the multicurve $\underline \alpha$. 
Denote by $\widetilde{C}$ the normalization of $C$ along with the node branches data.

\begin{prop}\label{isotmaxdeg} Let $x:\{ pt \} \to \overline{{\mathcal M}_g^{n \; an}}$ 
be a point representing an $n$-pointed stable curve $C_x$ and $B \Aut(C_x) \to  \overline{{\mathcal M}_g^{n \; an}}$ its residual gerbe. 
Let $U$ be a  neighborhood of $x$ such that $B \Aut(C_x)\to U$ is a deformation retract. Then 
there is a short  exact sequence: 

\[
1 \to \Z ^{\underline{\alpha}} \to \pi_1(U \cap  {\mathcal M}_g^{n \; an}, \ast) \to \Aut(C_x) \to 1,
\]
where $\Aut(C_x)$ acts through its natural action on the free abelian group  over 
the vertex set of the dual graph $\Gamma_x$ of $C_x$.  
Moreover, if $C_x$ is maximally degenerate, then  the exact sequence above is isomorphic to the sequence in Lemma \ref{lemma:stabilizer} for an appropriate pants decomposition. 
\end{prop}

Let $\mathbf{p}=(p_0,p_1,\ldots,p_r)$ be a ramification multi-index. 
If $C$ is a stable curve whose topological type is obtained from the multicurve $\underline{\alpha}$ 
let us define $\mu_{\mathbf{p}}(C)$ be the product $\mu_{p_i}^{k_i}$, where 
$k_i$ is the number of loops in $\alpha$ of type $i$. 

\begin{prop}\label{isotropygroup}
Consider a point $x[\mathbf{p} ]$ of $\overline{{\mathcal M}_g^{n \; an}}[\mathbf{p}]$ mapping to a point  $x$ of $\overline{{\mathcal M}_g^{n \; an}}$ which corresponds to the stable curve $C_x$. 
Then its isotropy group is given by the exact sequence:   
\[
1 \to \mu_{\mathbf{p}}(C_x) \to \pi_1(\overline{{\mathcal M}_g^{n \; an}}[\mathbf{p}], x[\mathbf{p}])_{loc} \to \Aut(C_x) \to 1.
\]

\end{prop}

\subsection{Faithfully representing isotropy groups}
The purpose of this section is to find the image of the isotropy group under suitable finite dimensional representations. 
Let $C$ be a stable curve of genus $g$ with $n$ marked points with a topological model obtained from a closed oriented surface $S$ with $n$ marked points by pinching 
every loop of a multicurve $\underline\alpha$ to a point.
Let $\Gamma=\Gamma(\underline\alpha)$ be the weighted dual graph of $\underline \alpha$. 
Denote by $\widetilde{C}$ the normalization of $C$ along with the node branches data. 
We then have an exact sequence: 
\[ 1\to \widetilde{\Aut}(\widetilde{C})\to \Aut(C)\stackrel{B}{\to} G\Aut(\Gamma(\underline\alpha))\to 1,\]
where $\widetilde{\Aut}(\widetilde{C})$ is the group of automorphisms of  $\widetilde{C}$
preserving each node branch and $G\Aut(\Gamma(\underline\alpha))$ is the subgroup of 
the group   $\Aut(\Gamma(\underline\alpha))$  of automorphisms of the graph which can be lifted to 
$C$.

Further let $\widetilde{I}(C)\subseteq \PMod(S)=\PMod(\Sigma_g^n)$ denote the lift of
the automorphism group of the stable curve $C$ to the mapping class group $\PMod(S)$, which fits into the exact sequence: 
\[ 1\to \Z^{\underline\alpha}\to \widetilde{I}(C)\stackrel{Q}\to \Aut(C)\to 1.\]

Let $\theta_{p, S}:\Mod(S)\to \Aut(H_1(S;\Z/p\Z))$ be the homology representation 
for the surface $S$ on the homology with $\Z/p\Z$ coefficients.

\begin{lemma}\label{torsiondetect}
For every stable curve $C$ with associated multicurve $\underline{\alpha}$ and $p\geq 3$  we have 
\[ \ker(B\circ Q)\cap \ker(\theta_{p,S})\subset  \Z^{\underline{\alpha}}.\]
\end{lemma}
\begin{proof}
Let $\widetilde{C}$ have the connected components $C_i$ which are of   
 genus $g_i$, $k_i$ marked points coming from the branched nodes and $r_i$ additional marked points inherited from the marked points of $C$. 
 Note that our initial surface $S$ is obtained from the union of 
surfaces $\Sigma^{r_i}_{g_i,k_i}$ by identifying boundary circles corresponding to the pinched 
curves from $\underline\alpha$. 
 
As the homological  representation  $\theta_{p,S}$ is functorial, 
the inclusion induces a homomorphism  $H_1(\Sigma_{g_i,k_i}^{r_i}; \Z/p\Z)  \to  H_1(S; \Z/p\Z)$ which is equivariant with respect to the natural map: 
\[\Phi_i: \Mod(\Sigma_{g_i,k_i}^{r_i})  \longrightarrow  \Mod(S,\underline{\alpha}). 
\]
Set further: 
 \[ \Phi=\prod_i\Phi_i : \prod_i \Mod(\Sigma_{g_i,k_i}^{r_i}) \to \Mod(S,\underline{\alpha}).\]

Let $\Lambda \subset H_1(S, \Z/p\Z)$ and $\Lambda_i\subset H_1(\Sigma_{g_i,k_i}^{r_i}; \Z/p\Z)$ be the subgroups generated by the homology classes of the components of the multicurves  
$\underline{\alpha}$ and $\underline{\alpha}_i=\underline{\alpha}\cap \Sigma_{g_i,k_i}^{r_i}$, respectively. These are  isotropic subspaces for the intersection form which 
are invariant by the image of $\Phi$ and  $\Mod(\Sigma_{g_i,k_i}^{r_i})$, respectively. 
Then, there is an injective homomorphism 
\[ I: \bigoplus_{i} H_1(\Sigma_{g_i,k_i}^{r_i}; \Z/p\Z)/\Lambda_i \longrightarrow H_1(S; \Z/p\Z)/ \Lambda,\]
which is equivariant under the group homomorphism $\Phi$, 
where the groups act via the homology representation.

Note that the action of the subgroup $\Z^{\underline{\alpha}_i}$ on 
$H_1(\Sigma_{g_i,k_i}^{r_i}; \Z/p\Z)/\Lambda_i$  is trivial. Therefore the action  
of  $\Mod(\Sigma_{g_i,k_i}^{r_i})$ on  $H_1(\Sigma_{g_i,k_i}^{r_i}; \Z/p\Z)$
induces an action of  $\Mod(\Sigma_{g_i}^{k_i+r_i})$ on  the quotient $H_1(\Sigma_{g_i,k_i}^{r_i}; \Z/p\Z)/\Lambda_i$. The latter action coincides, under the isomorphism induced by the inclusion  
$H_1(\Sigma_{g_i,k_i}^{r_i}; \Z/p\Z)\to H_1(\Sigma_{g_i}^{k_i+r_i}; \Z/p\Z)$, 
with the reduced homology representation of $\Mod(\Sigma_{g_i}^{k_i+r_i})$ on  $H_1(\Sigma_{g_i}^{k_i+r_i}; \Z/p\Z)/\Lambda_i$. Note that the action of 
the pure mapping class group $\PMod(\Sigma_{g_i}^{k_i+r_i})$ on $\Lambda_i$ is trivial. 

Although the elements of $\widetilde{I}(C)$ are not in the image of $\Phi$, they do preserve $\Lambda$. Observe that the action of the subgroup $\Z^{\underline{\alpha}}\subset \widetilde{I}(C)$ on $ H_1(S; \Z/p\Z)/ \Lambda$ is trivial, as well. 
By what precedes, the group $\widetilde{I}(C)$  acts on $\bigoplus_{i} H_1(\Sigma_{g_i,k_i}^{r_i}; \Z/p\Z)/\Lambda_i$ via its quotient $\Aut(C)\subset \Aut(\widetilde{C})$. Observe that the map $I$ is $\widetilde{I}(C)$-equivariant. 

Furthermore there is natural map $\ker(B\circ Q) \to \prod_i \Aut(C_i)$. For $\phi\in \ker(B\circ Q)$ we denote 
by $\phi_i$ the corresponding automorphism of $C_i$ viewed as a mapping class in $\PMod(\Sigma_{g_i}^{k_i+r_i})$. 

The well-known Serre lemma states that the action of 
a nontrivial finite order element of $\Mod(\Sigma_{g}^{n})$ 
on $H_1(\Sigma_{g}^{n})$ is nontrivial, provided $\Sigma_{g}^{n}$ is hyperbolic, i.e. 
$2g-2+n>0$.  Since the curve $C$ is stable, each component $C_i$ should be hyperbolic.
If  $\theta_{p,S}(\phi)=1$, then $\phi_i$ acts trivially on 
$H_1(\Sigma_{g_i}^{k_i+r_i}; \Z/p\Z)/\Lambda_i$. As $\phi_i$ is 
a pure mapping class, its homological representation  
on $H_1(\Sigma_{g_i}^{k_i+r_i}; \Z/p\Z)$ is trivial as well. 
Therefore, Serre's lemma enables to conclude that $\phi_i= 1$.

\end{proof}

\begin{rem}
The map $I$ extends naturally to an injective homomorphism:  
\[ I: \bigoplus_{i} H_1(\Sigma_{g_i,k_i}^{r_i}, \partial \Sigma_{g_i,k_i}^{r_i}; \Z/p\Z) \longrightarrow H_1(S,\underline{\alpha}; \Z/p\Z)\]
which is equivariant under the group homomorphism $\Phi$ and fits into the exact sequence: 
\[ 1\to \bigoplus_{i} H_1(\Sigma_{g_i,k_i}^{r_i}, \partial \Sigma_{g_i,k_i}^{r_i}; \Z/p\Z) \stackrel{I}{\longrightarrow} H_1(S,\underline{\alpha}; \Z/p\Z)\to H_1(\Gamma(\underline{\alpha});\Z/p\Z)\to 1.\]
\end{rem}

\begin{lemma}\label{graphdetect} 
Let $C$ be a stable curve with $n$ marked points with associated multicurve $\underline{\alpha}$, 
$x\in \widetilde{I}(C)$ and $p\geq 3$. If $B\circ Q (x)\neq 1\in \Aut(\Gamma(\underline\alpha))$,
then  $\rho_{p, {\mathbf i}}(x)\neq 1$, where $\mathbf i=(i_1, \ldots, i_n)$ is any coloring of the marked points
such that there exists a coloring $(j_1, \ldots, j_q)$ of the nodes 
of $C$ so that for all $i$ $W((\Sigma_{g_i,r_i+k_i})_{(\mathbf{i}_i,\mathbf{j}_i)})\not =0$  where $(\mathbf{i})_i$ is the restriction of $(\mathbf{i})$ to the $r_i$ marked points in 
$\Sigma_{g_i,r_i+k_i}$ and $(\mathbf{j}_i)$ is the restriction of $(\mathbf{j})$ to the $k_i$ nodes in 
$\Sigma_{g_i,r_i+k_i}$. 

\end{lemma}
\begin{proof}
We can choose basis of the space of conformal blocks associated to 
$S$ corresponding to a pants decomposition containing the multicurve $\mathbf\alpha$. 
Then the action of $\rho_{p}(x)$ on the conformal blocks permutes 
the conformal blocks associated to the connected components of $C$ minus the nodes according to 
the nontrivial action of $B\circ Q (x)$. 
\end{proof}

\begin{proposition}\label{faithfulisotropy}
Let $\mathbf i=(2,2,\ldots,2)$, $g\geq 2, n\geq 0$, $(g,n)\not=(2,0)$
$2g-2+n>0$.
\begin{enumerate}
\item If $p\geq 5$ is odd, then $\rho_{p,(\mathbf i)}\oplus \theta_{p, \Sigma_{g}^n}$ 
factors through $\Mod(\Sigma_g^n)/\Mod(\Sigma_g^n)[p]$ and sends 
$\pi_1(\overline{{\mathcal M}_g^{n \; an}}[p], x[p])_{loc}$ isomorphically onto its image. 
\item If   $p\geq 10$ is even and $(g,n,p)\not=(2,0,12)$ then  the representation $\rho_{p,(\mathbf i)}\oplus \theta_{2p, \Sigma_{g}^n}$ factors through
$\Mod(\Sigma_g^n)/\Mod(\Sigma_g^n)[2p,  p/{\rm g.c.d.}(p,4)]$ and sends the isotropy group 
$\pi_1(\overline{{\mathcal M}_g^{n \; an}}[2p,  p/{\rm g.c.d.}(p,4)], x[2p,  p/{\rm g.c.d.}(p,4)])_{loc}$ isomorphically onto its image. 
\end{enumerate}
\end{proposition}
\begin{proof}
It is well-known that $\ker\theta_p$ contains both $T_{\gamma}^p$, for nonseparating $\gamma$ 
and $T_{\gamma}$, when $\gamma$ is separating. Then  Proposition \ref{factori} implies the 
claims concerning the factorization of the given representations.
From Lemmas \ref{torsiondetect} and \ref{graphdetect} the kernel of these representations is  
contained within $\Z^{\underline{\alpha}}$. Then the   
faithfulness is a consequence of Proposition \ref{prop:stabilizer} which describes the 
kernel on $\Z^{\underline{\alpha}}$ and Proposition \ref{isotropygroup} which describes 
the isotropy group. 
\end{proof}

\subsection{$\overline{{\mathcal M}_g^{n \; an}}[p]$ is uniformizable}
The last step in the proof of Theorem \ref{kahler} is to show that 
the orbifolds we consider are uniformizable.

\begin{prop}\label{prop:uniformizable} Assume  $g\geq 2, n\geq 0$ and
$2g-2+n>0$.

If $p\geq 5$ is odd, then $\overline{{\mathcal M}_g^{n \; an}}[p]$ is uniformizable. 
Further, for even $p\geq 10$, and $(g,n,p)\not=(2,0,12)$,   the orbifold 
$\overline{{\mathcal M}_g^{n \; an}}[2p,  p/{\rm g.c.d.}(p,4)]$ is uniformizable. 
\end{prop}
\begin{proof}
Let $p$ be odd. 
Proposition \ref{faithfulisotropy} provided a projective representation of 
$\pi_1(\overline{{\mathcal M}_g^{n \; an}}[p])$ which is faithful on the isotropy subgroups.
This can be easily converted into a linear representation by first composing 
it with the adjoint representation of the projective unitary group, as the later is centerfree.   
By Proposition \ref{infty-closed} the image of the linear representation is infinite and
by the Selberg lemma it has a finite index torsionfree subgroup $J$. By taking the quotient of the image by $J$ we obtain a homomorphism 
from $\pi_1(\overline{{\mathcal M}_g^{n \; an}}[p])$ into a finite group $G$ in which all 
isotropy groups should inject. Therefore $\overline{{\mathcal M}_g^{n \; an}}[p]$ is uniformizable, as claimed. When $p$ is even, the proof is similar. 
\end{proof}

\begin{rem}
This gives an alternate and perhaps  simpler version
 of the existence of smooth
finite Galois coverings of the moduli space of stable curves due to Looijenga, Pikaart and Boggi. Letting $p\to \infty$,  we also recover the universal ramification of the
Teichm\"uller tower (\cite{Bry}, corrected in \cite{Loo}).
\end{rem}

\subsection{End of proof of Theorem \ref{kahler}}
Proposition \ref{prop:uniformizable} and an argument from \cite{vjm} implies that 
$\pi_1(\overline{{\mathcal M}_g^{n \; an}}[p])$ is a K\"ahler group. 
In fact, consider a Galois \'etale covering 
$M\to \overline{{\mathcal M}_g^{n \; an}}[p]$ with deck group $G$. 
Then $M$ can be chosen to be a smooth projective variety. On the other hand there  
exists a projective surface $Y$ such that $\pi_1(Y)=G$, as $G$ is finite. 
If $\widetilde{Y}$ denotes its universal covering, then $\pi_1(\overline{{\mathcal M}_g^{n \; an}}[p])$ acts diagonally on $\widetilde{M}\times Y$ 
freely and properly discontinuously. Then the quotient is a smooth projective variety 
having fundamental group $\pi_1(\overline{{\mathcal M}_g^{n \; an}}[p])$. 
The case of even $p$ is similar. 

This indeed shows that the groups $\Mod(\Sigma_g^n)/\Mod(\Sigma_g^n)[p]$, for odd $p\geq 5$, and 
respectively $\Mod(\Sigma_g^n)/\Mod(\Sigma_g^n)[2p,  p/{\rm g.c.d.}(p,4)]$, for even $p\geq 10$ 
are K\"ahler and ends the proof of Theorem \ref{kahler}.

\section{Bogomolov-Katzarkov Surfaces}

\subsection{} 
 
Let $\pi: S \to C$ be a non isotrivial fibration of a smooth complex  projective algebraic  surface $S$ onto a smooth curve $C$ such that:
\begin{enumerate}
 \item $\pi$ is a projective  morphism with connected fibres, 
 \item for generic $c\in C$, the fiber $\pi ^{-1}(c)$ is a genus $g\ge 2$ curve
 \item the singular fibres of $\pi$ are reduced stable curves, in particular the singular set consists of  ordinary double points.
\end{enumerate}

In a transcendental neighborhood 
of  $\sigma\in S$ a node of a singular fiber there are complex coordinates $(x,y)$ on $S$ and a complex coordinate $z$ on $C$ near $\pi(\sigma)$
such that  $z(\pi(x,y))=xy$. 

By the definition of the stack $\Mgbar$ of genus $g$ stable curves,  
we have a map $\psi: C\to \Mgbar$, which induces a holomorphic map  
$C^{an} \to \Mgbaran$. By construction this map is covered by a map of stacks $S\to \Mgobar$ and we have a 
2-commutative diagram: 
\[
\xymatrix{S\ar[d]_{\pi}\ar[r] & \Mgobar \ar[d]^{f_1}\ar@{=>}[ld]\\
C\ar[r]^{\psi} & \Mgbar
}
\]
 
The square is cartesian, i.e.: $S\simeq C\times_{\psi,f_1} \Mgobar$,   $\pi$ being equivalent to $\psi^* f_1$.

Fix an integer $N\ge 2$. If $N=2$ and $g(C)=0$ assume $\pi$ has at least 4 singular fibers, in general there at least 3 singular fibers. We will mostly be interested in the case $N$ is large enough, meaning $N\ge 5$ if $N$ is odd.

\subsection{ The stack $\mathcal{S}_{\pi}[\mathbf k]$} 
For every multi-index $\mathbf k$ we define the following stack: 
\[
\mathcal{S}_{\pi}[\mathbf k]:=S^{an}\times_{\Mgobaran}\Mgobaran[\mathbf k].
\]
According to the convention from the introduction we write 
$\mathcal{S}_{\pi}[N]$ for $\mathcal{S}_{\pi}[N, N,N,\ldots,N]$  
and  $\mathcal{S}_{\pi}[N,L]$ for $\mathcal{S}_{\pi}[N, L,L,\ldots,L]$.

\begin{prop} \label{boka} 
The  stack $\mathcal{S}_{\pi}[N]$ is a smooth compact K\"ahler orbifold with moduli space $S$. 
There is a  proper map $\mathcal{S}_{\pi}[N] \to \mathcal{C}[N]$ where $\mathcal{C}[N]$ is a non-elliptic orbicurve 
 whose general fiber is a smooth curve of genus $g$ giving rise to a fibration exact sequence :
\[ 
 1 \to  I_N(\mathbf{C}_\pi)  \to
 \pi_1(\mathcal{S}_{\pi}[N])\to  \pi_1(\mathcal{C}[N])\to 1. 
\]
where $I_N(\mathbf{C}_\pi)$ is the image of the fundamental group of a generic fiber into 
$\pi_1(\mathcal{S}_{\pi}[N]))$. Furthermore:
$$I_N(\mathbf{C}_\pi)\cong\frac{\pi_1(\Sigma_g) }{\langle\langle c^N, \ c\in  \mathbf{C}_\pi\rangle\rangle}
$$
where $\mathbf{C}_\pi$ is a monodromy invariant family of conjugacy classes in $\pi_1(\Sigma_g) $ representing free homotopy classes of simple closed oriented curves. 
\end{prop}

\begin{proof}
 This is in effect  proved in \cite{BK}.  For the convenience of the reader we will 
 rephrase their paper in our terminology. 
 
The smoothness of $S$ translates into the statement
that given an uniformizing parameter  $z$ for $C$ at $q\in C$ and $\eta$ a chart adapted to $\mathcal{D}$ centered at $\psi(q)$, 
$\eta^{-1} \psi$ assumes the form $z\mapsto (z_1, \dots, z_{3g-3})$ with $\frac{d z_i}{dz}(q) \not =0$ for $1\le i \le k(\eta)=\# Sing(\pi^{-1}(p))$. 
Furthermore, near the image $Q$ of a singular point $q'$ of $\pi^{-1}(q)$ we can find an \'etale chart of $\Mgobaran$ at $Q$ such that the map $S \to \Mgobaran$ takes the form
$(x,y)\mapsto (x, y, z_2(x,y), \ldots, z_{3g-3}(x,y))$ in local coordinates $(x,y)$ of $S$ near $q'$, whereas the map to $C$ is given by $t=xy$. 
Using the \'etale charts in \cite{DMu} in which $f_1$ takes the form $(\xi_1,\eta_1, z_2, \ldots, z_{3g-3})$ we also can arrange to have $z_k(x,y)=z_k(xy)$ since $\pi\simeq\psi^* f_1$ . 

So the map from $S$ to $\Mgobaran$ is transverse to $\mathcal{D}_1$ and the pull back of $\mathcal{D}_1$ is just the sum $\Sigma:=\sum_{q\in B} S_q$ of the (reduced) singular fibres of $\pi$, $B$
being the set of critical values of $\pi$.
Certainly the formation of the smooth root stack of a normal crossing divisor commutes with transversal mappings and we get:
\[ \mathcal{S}_{\pi}[N] \simeq S[N,\Sigma].
\]
 
 There is a natural map $\mathcal{S}_{\pi}[N] \to \mathcal{C}[N]$ where $\mathcal{C}[N]=C^{an}[\sqrt[N]{B}]$ according  to (\cite{cadman2007}, Rem.2.2.2). Actually
$S^{an}\times_{C^{an}} \mathcal{C}[N] \simeq S[\sqrt[N]{\Sigma}]$ and $\mathcal{S}_{\pi}[N]  \to S[\sqrt[N]{\Sigma}]$ is the canonical stack. Indeed the germ $S[\sqrt[N]{\Sigma}]_{q'}$
is equivalent to the germ at the origin of the stack with a cyclic quotient singularity of type $\frac{1}{N} (1,-1)$: $$[\mu_N \backslash A_N]= [\mu_N\backslash \{t^N=xy\}].$$

Now $\mathcal{C}[N]$ is an uniformizable orbicurve since otherwise $C$ were rational and $\pi$ would have $\le 2$ singular fibers, but  this situation is excluded by the big Picard theorem and the fact that the period domain for ppav is a bounded domain. Let $C_N$ be a complex projective curve with an action of a finite group $G$ such that $[G\backslash C_N]=\mathcal{C}[N]$, namely the Galois cover $C_N\to C$  ramifies precisely over $B$ with index $N$. 

Consider now the following {\em singular surface} $S_N:=S\times_C C_N$. Then $S_N$ is a normal surface with $A_{N-1}$ singularities and the canonical stack $(S\times_C C_N)^{can}$  is an \'etale covering of $\mathcal{S}_{\pi}[N]$. Actually
since $(S\times_C C_N)^{can}$ is smooth its fundamental group is isomorphic to the fundamental group of the open substack obtained by deleting the finite substack of its orbifold points, 
which is then precisely the open surface considered in \cite{BK}. There $C_N$ is denoted by $R$, $S\times_C C_N$ by $S$ and the singular set of $S\times_C C_N$ by $Q$ which is in a natural bijection withe the sets of nodes.  

Let $D$ be a small coordinate disk in $C_N$ centered at a singular fiber of $S_N\to C_N$ which is a degree $N$  covering of a small coordinate disk $D'$ 
of $C$ centered at a critical point of $\pi$, such that $D$ only ramifies over $D'$ at that critical point. Let $T'$ be a non singular
fiber of $\pi:S_N\times_{C_N} D \to D$ identified with a non singular fiber of $\pi^{-1} (D')$. 

Then there is a collection $c'_1, \ldots, c'_q$ of disjoint simple curves in $T'$ which appear as geometric vanishing cycles in  a Clemens contraction  for the 
family of stable curves $S\times_C D' \to D'$. Denote by $\bar c_1, \ldots, \bar c_q$ their conjugacy classes in $\pi_1(T')$. Introduce
base points denoted by $\ast'$ which are preserved by the  maps under consideration. 
 
\begin{lem} \label{presloc}
The inclusion map $(T',\ast')\hookrightarrow ((S_N\times_{C_N} D)^{can}, \ast')$ induces at fundamental group level 
the quotient group homomorphism onto:
\[\pi_1((S_N\times_{C_N} D)^{can},\ast')=\pi_1(T',\ast')/\langle\langle\bar c_1^N, \dots, \bar c_q^N\rangle\rangle,\] 
where $\langle\langle\_\rangle\rangle$ means the normal subgroup generated by the respective elements. 
\end{lem}
For the sake of simplicity we kept the notation $\pi_1(T',\ast')$ instead of the more appropriate $\pi_1(T',\ast_{T'})$.
\begin{proof}
This  is just a reformulation of (\cite{BK}, Thm.2.1), using Van Kampen  for stacks.
\end{proof}

One has $g(C_N)\ge 1$ and we denote  the universal covering  $\widetilde{C_N ^{univ}}\cong \C \ \mathrm{or} \  \Delta$  by $\mathbb E$ to emphasize  it is homeomorphic to an euclidian plane. 
Consider the canonical stack 
\[ \mathcal{T}:=(S_N\times_{C_N} \mathbb E)^{can} \simeq S_N^{can} \times_{C_N}\mathbb E \] 
and introduce a coherent choice of base points $\ast$. We can assume that 
$\ast_\mathcal{T}$ does not lie on a singular fiber of the projection to $\mathbb E$ and has consequently a trivial inertia group.
Let $T$ be the fiber of $\pi_{\mathbb E}:=\pi\times_{C} 1_{\mathbb E}:\mathcal{T} \to \mathbb E$ at $\ast_{\mathbb E}$. 

The set of branch points $B(\mathbb E)$ of $\pi_{\mathbb E}$ is the preimage 
of branch points of $\pi$ and thus  it is discrete in $\mathbb E$.
Hence there is a 
countable discrete subset  $J\subset ]0, +\infty[$
such that $\partial \mathbb{B}_{\mathbb E}(\ast, R)$ does not intersect  $B(\mathbb E)$ when $0<R\not\in J$. For each such $R$ we can construct a tree $\mathcal{A}_R$ rooted at $\ast$ in $\mathbb{B}_{\mathbb E}(\ast, R)$ and ending at 
the various points of $B(\mathbb E) \cap \mathbb{B}_{\mathbb E}(\ast, R)$. Obviously
we can arrange so that $R'\ge R \Rightarrow \mathcal{A}_R \subset \mathcal{A}_{R'}$ with equality between two consecutive values of 
$J$. We may also assume that the lifts of the various
$\ast'$ lie on $\mathcal{A}_R$ . Using Ehresmann lemma (and an horizontal normal bundle to 
$\pi_{\mathbb E}|_{\mathbb E\setminus B(\mathbb E)}$ along the geodesic rays
of $\mathcal{A}_R$ emanating from $\ast_{\mathbb E}$ we can transport the geometric vanishing cycles at each $\ast'$ to  
a system of disjoint simple curves in $T$. Doing this for all $\ast'$  yields a family $\{ {c}_{\alpha}\}_{\alpha \in A(R)}$ of 
closed simple curves in $T$. 
This yields a family $\{ \bar c_{\alpha}\}_{\alpha \in A(R)}$ of conjugacy classes in $\pi_1(T,\ast)$ and $R \mapsto A(R)$ is an 
increasing function. Taking the union for all $R$
we get a family $\mathbf{C}_{\pi}=\{ \bar c_{\alpha}\}_{\alpha \in A}$ of conjugacy classes in $\pi_1(T,\ast)$.

\begin{lem} We have 
\[\pi_1(\pi_{\mathbb E}^{-1} (\mathbb{B}_{\mathbb E}(\ast,R)),\ast)=\pi_1(T,\ast)/\langle\langle \bar c_{\alpha}^N , \ \alpha \in A(R) \rangle\rangle.
\]
Moreover, the inclusion map $(T,\ast)\hookrightarrow (\pi_{\mathbb E}^{-1} (\mathbb{B}_{\mathbb E}(\ast,R)),\ast)$
induces the quotient group homomorphism.  In particular, when $R=\infty$, we obtain 
\[\pi_1(\mathcal{T},\ast)=\pi_1(T,\ast)/\langle\langle \bar c_{\alpha}^N , \ \alpha \in A \rangle\rangle.\]
 \end{lem}
 \begin{proof} Van Kampen  for stacks.
 \end{proof}
 
It follows from the covering theory for topological stacks  (\cite{No2}) that 
there is an exact sequence:
\[
1 \to \pi_1(\mathcal T, \ast) \to \pi_1(S_N^{can}, \ast) \to \pi_1( C_N ,\ast) \to 1,
\]
the normal subgroup being the image of $\pi_1(T,\ast)$ and similarly
\[
1 \to \pi_1(\mathcal T, \ast) \to \pi_1(\mathcal{S}[N], \ast) \to \pi_1(\mathcal{C}[N],\ast) \to 1. 
\]
We set therefore $I_N(\mathbf{C}_{\pi})=\pi_1(\mathcal T, \ast)$. 
This completes the proof of Proposition \ref{boka}. 
\end{proof}
  
\subsection{The fiber group $I_{N}(\mathbf{C}_{\pi})$} 

As argued in \cite{BK}, the fiber group  is a very interesting object. Consider the fiber exact sequence of the proper smooth map obtained from $\pi$ by deleting the singular fibers:

\[
1 \to \pi_1(\Sigma_g)) \to \pi_1(S\setminus \pi^{-1} (B)) \to \pi_1(C\setminus B) \to 1
\]

Consider $\rho:\pi_1(C\setminus B) \to \Out^{+}(\pi_1(\Sigma_g))$ and $M_{\pi}=Im(\rho)$ the monodromy group.
 The group $M_{\pi}$ acts on conjugacy classes of elements of $\pi_1(\Sigma_g)$. It is a non trivial invariant of the intriguing combinatorial object $M_{\pi} \curvearrowright \mathbf{C}_{\pi}$.

\begin{lem}
 $\mathbf{C}_{\pi}$ is a finite union of $M_{\pi}$-orbits. 
\end{lem}

\begin{proof}
This is a reformulation of the discussion before (\cite{BK}, Def. 3.1).
\end{proof}

Apart from that,  few things seem to be known on the 
the actions $M_{\pi} \curvearrowright \mathbf{C}_{\pi}$ which can occur in algebraic families of curves, except the restrictions coming from 
Deligne's semisimplicity theorem and also his theorem on the fixed part
when one uses the natural map $M_{\pi}\curvearrowright \mathbf{C}_{\pi} \to M_{\pi, H_1}\curvearrowright  H_1(\Sigma_g,\Z)$. 

However there is a surjective morphism 
\[\pi_1(\Sigma_g) \to B_N(\Sigma_g)=\pi_1(\Sigma_g) \slash \langle\langle x^N , x \in \pi_1(\Sigma_g)\rangle\rangle\] onto its $N$-th verbal quotient, also called the $N$-th Burnside group of the surface. 
The Restricted Burnside Problem, solved affirmatively by Zelmanov, states that 
the profinite completion of the Burnside group (of a free group) is finite. 
For  large enough $N$, 
$B_N(\Sigma_g)$ is infinite having the $N$-th Burnside group of the free group on $g$ generators as a quotient and hence it is 
not residually finite.
This won't  imply that $\pi_1(\mathcal{S}[N])$ is not residually finite but it implies that
the fiber group is  infinite. This would be a very desirable 
situation since the only \lq obvious\rq \  representations of  $\pi_1(\mathcal{S}[N])$ come from $\pi_1(\mathcal{C}[N])$.

Passing to $B_N(\Sigma_g)$ is an extreme step. It is better to look at the Burnside type quotient by the $N$-th powers of the primitive elements. Indeed, this quotient occurs as a fiber group if 
the monodromy group of $\pi$ is the mapping class group and every topological type of simple closed curve occurs in $\mathbf{C}_{\pi}$, a case that can be realized.  The following proposition 
makes precise some infiniteness statements in \cite{BK}.

\begin{prop}\label{inffibgp}
The fiber group $I_N(\mathbf{C}_{\pi})$ is  infinite when $g=2$ and $N\geq 4$ or $g\geq 3$ and $N\not\in\{2,3,4,6,8,12\}$. 
\end{prop}

\begin{proof}
This is a corollary of the proof of \cite[Cor. 4]{AF}, the main ingredient being  \cite{Hum} when $g=2$ and the Koberda-Santharoubane theorem in the version of \cite[Proof of Proposition 3.2]{LF}. 
\end{proof}

\subsection{The surface $S(N)$} 

\subsubsection{$N$ odd}

 \begin{prop} \label{unifbk} If $N$ is odd,  the stack 
  $\mathcal{S}_{\pi}[N]$,  is a uniformizable  compact K\"ahler orbifold with a projective moduli space.
  
\end{prop}

\begin{proof}
This is \cite[ Lemma 2.7]{BK}, which uses  a very specific 
Teichm\"uller level.  First, they  check that $(S_N\times_{C_N} D)^{can}$ is uniformizable. 
 
Let us describe their argument as \cite{BK} skips some details.  
Set $\Pi=\pi_1(T')$, for the fundamental group of the generic fiber, which is isomorphic to $\pi_1(\Sigma_g)$, $\Pi_{(3)}=[[\Pi,\Pi],\Pi]$ for the third term of the lower central series 
and $\Pi^N=\langle\langle x^N, x\in \Pi\rangle\rangle$ be the normal subgroup generated by the 
$N$-th powers. Consider next the quotient 
$U_N(\Pi)=\Pi/\Pi_{(3)}\cdot \Pi^N$. 

The Lie algebra  of $\Pi/\Pi_{(3)}$ is the quotient of 
$\Z[A_i, B_i, E_{l,k}, F_{m,n}, G_{m,n}\}]_{1\le i \le g, 1\le l\le k \le g, 1\le m<n\le g}$ 
by the relations 
$[A_l,B_k]=E_{l,k}$ 
$[A_m,A_n]=F_{m,n}$ $[B_m,B_n]=G_{m,n}$  where $E_{l,k}$ $F_{m,n}$ $G_{m,n}$ are central, and  $ \sum_{k=1}^g E_{k,k}=0$. 
Moreover, it defines by means of the Baker-Campbell-Hausdorff formula an 
unipotent group scheme $U$ which is smooth over $\Z[1/2]$, so that $U(\Z)=\Pi/\Pi_{(3)}$. 
There is a map $\Pi\to U_N(\Pi)\to U(\Z/N\Z)$  sending the classes $\bar c_i$ to elements of order $N$ 
in $U(\Z/N\Z)$. In fact nontrivial elements of the nilpotent group  $U(\Z/N\Z)$ have order $N$ and 
thus  $U(\Z/N\Z)$ is finite. Then Lemma \ref{presloc} implies that $(S_N\times_{C_N} D)^{can}$ is uniformizable.

Observe that the subgroup $H_N=\Pi_{(3)}\cdot \Pi^N$ is  characteristic, i.e. invariant by all the automorphisms of $\Pi$ and of finite index. Hence the same
quotient of $\pi_1(\Sigma_g)$ works for all singular fibers.  

Since $B\neq \emptyset$, the group $\pi_1(C\setminus B)$ is free and so  $\rho$ lifts to $\Aut^+(\pi_1(\Sigma_g))$. Moreover, 
the fundamental group of $S\setminus \pi^{-1}(B)$ is a semidirect product $\pi_1(\Sigma_g)\rtimes \pi_1(C\setminus B)$. Since the automorphism group of a finite group is finite,
there is a normal subgroup of finite index $H\lhd \pi_1(C\setminus B)$ which acts trivially 
on $U_N(\pi_1(\Sigma_g))$. Then,  $H_N\rtimes H$ is a finite index subgroup of $\pi_1(S\setminus \pi^{-1} (B))$.
Moreover, $\pi_1(S\setminus \pi^{-1} (B))$ surjects onto 
the finite group  $U_N(\pi_1(\Sigma_g)) \rtimes (\pi_1(C\setminus B)\slash H)$. 
Note that this surjective homomorphism maps the elements of $\mathbf{C}_{\pi}$ to elements of order $N$. In particular, $K_N (\mathbf{C}_{\pi}):=\ker(\pi_1(\Sigma_g) \to I_N(\mathbf{C}_{\pi})) \subseteq H_N$. 

Since $H_N$ is a characteristic subgroup of $\pi_1(\Sigma_g)$ it is also a normal subgroup of $\pi_1(S\setminus \pi^{-1}(B)) $ contained in $\pi_1(\Sigma_g)$.  
We then obtain the following diagram, keeping
in mind the natural surjective morphism $\pi_1(S_N\setminus \pi^{-1}(B_N))\to \pi_1(S_N^{can})$: 

\[\xymatrix{1\ar[r] & \pi_1(\Sigma_g)  \ar[r] & \pi_1(S\setminus \pi^{-1}(B)) \ar[r] & \pi_1(C\setminus B) \ar[r] & 1\\
1\ar[r] & \pi_1(\Sigma_g) \ar@{>>}[d]\ar@{=}[u] \ar[r] & \pi_1(S_N\setminus \pi^{-1}(B_N)) \ar[r]\ar@{>>}[d]\ar@{>->}[u] & \pi_1(C_N\setminus B_N) \ar@{=}[d]\ar[r] \ar@{>->}[u]& 1\\
 1 \ar[r] &  \pi_1(\Sigma_g) \slash H_N \ar@{=}[d] \ar[r]  &  \pi_1(S_N\setminus \pi^{-1}(B_N))/H_N \ar[r] \ar@{>>}[d] &   \pi_1(C_N\setminus B_N) \ar[r] \ar@{>>}[d]  & 1\\
  1 \ar[r] &  \pi_1(\Sigma_g) \slash H_N  \ar[r]  &  \pi_1(S_N^{can})/{\rm Im}(H_N) \ar[r]  &   \pi_1(C_N) \ar[r]  & 1\\
   1 \ar[r] &  \pi_1(\Sigma_g) \slash K_N(\mathbf{C}_{\pi} ) \ar[r]\ar@{>>}[u]  &  \pi_1(S_N^{can}) \ar[r] \ar@{>>}[u] &   \pi_1(C_N) \ar[r] \ar@{=}[u] & 1.}
   \]
   
Consider the intersection $H'$ of $ H$ with $\pi_1(C_N\setminus B)$. 
Then $H_N \rtimes H'$ is a finite index normal subgroup of $\pi_1(S_N\setminus \pi^{-1}(B_N))$. Projecting it down to 
   $\pi_1(S_N\setminus \pi^{-1}(B_N))/H_N$ we get a normal subgroup of finite index $H''<  \pi_1(S_N\setminus \pi^{-1}(B_N))/H_N$ intersecting $\pi_1(\Sigma_g) \slash H_N $ trivially. 
   Projecting it down to 
   $ \pi_1(S_N^{can})/{\rm Im}(H_N)$ we claim that we get a normal subgroup of finite index $H'''<  \pi_1(S_N^{can})/{\rm Im}(H_N)$ intersecting $\pi_1(\Sigma_g) \slash H_N $ trivially. 
 To this purpose we need the following elementary lemma:    
  \begin{lem}\label{label:elementary}
   Let $A,B$ be groups and $B\to \Aut(A)$ a group homomorphism. Let $C\triangleleft B$ be a normal subgroup acting trivially on $A$. Then $A \rtimes B/C$ is a quotient of $A \rtimes B$. 
  \end{lem}
Our claim is then the consequence of Lemma \ref{label:elementary} applied to $C=\ker (H' \to  \pi_1(C_N))$,  $B=H'$, $A=\pi_1(\Sigma_g) \slash H_N$ setting $H'''= \{1 \}\times B/C$ and observing $A \times B/C$ is the image of $A\times H'\triangleleft \pi_1(S_N\setminus \pi^{-1}(B_N))/H_N $ a finite index normal subgroup of  $\pi_1(S_N^{can})/{\rm Im}(H_N)$. 
   
  Eventually we obtain a  homomorphism $\pi_1(S_N^{can}) \to G$ where $G$ is finite such that $\pi_1(\Sigma_g) \slash I_N(\mathbf{C}_{\pi} )$ is sent to a subgroup isomorphic to $U_N(\pi_1(\Sigma_g))$ by a surjective morphism
   mapping the elements in $\mathbf{C}_{\pi}$ to elements of order $N$. 
   
   This concludes the proof of Proposition \ref{unifbk}.  
\end{proof}
 
\begin{rem}
 If $\Mgobaran[N]$ is uniformizable then $\mathcal{S}_{\pi}[N]$ is uniformizable too. 
\end{rem}

We then denote by $S(N)=S_{\pi}(N)$, the $\pi$ being dropped if the context allows,  the smooth projective surface uniformizing the stack $\mathcal{S}_{\pi}[N]$, namely 
such that $[G\backslash S(N)]=S_{\pi}[N]$. Actually this does not define a smooth projective surface but a commensurability class\footnote{Two projective varieties
are {\em commensurable} if and only if they have a common finite \'etale covering. } of such surfaces.
 
\subsubsection{$N$ even}
 
The result from \cite{BK} should be restated as: 

\begin{prop} \label{unifbkeven} If $N\equiv 0 \ ({\rm mod}\, 4)$,  the stack 
  $\mathcal{S}_{\pi}[N,N/2]$  is a uniformizable  compact K\"ahler orbifold with a projective moduli space.  
\end{prop}
\begin{proof}
This is the result of \cite[Lemma 2.7]{BK} slightly modified in the even case. 
In fact, the order of a separating primitive element in $\frac{\pi_1(\Sigma_g)}{  \pi_1(\Sigma_g)_{(3)} \cdot \pi_1(\Sigma_g)^N}$ is  actually $\frac{N}{2}$  
(see \cite[Lemma 6.3]{PiJo}) and not $N$, correcting \cite[Lemma 2.5]{BK}. All arguments 
used above in the odd case are then valid. Alternatively, we can derive it from the 
uniformisability of  $\Mgobaran[N,N/2]$, which follow the same way as  
the uniformisability of  $\Mgbaran[N,N/2]$ from \cite[Thm.3.1.1]{PiJo}.  
\end{proof}
If $\mathcal{X}$ is a compact complex orbifold  (or an analytification of a Deligne-Mumford stack globally of finite type, say with a 
quasiprojective moduli space) we can do the following construction. 
Consider $H\triangleleft \pi_1(\mathcal{X})$ be a finite index subgroup such that the kernels of the isotropy morphisms $\pi_1^{loc}(\mathcal{X},x) \to \pi_1(\mathcal{X})/H$ 
coincide with the kernels of their universal counterpart $\pi_1^{loc}(\mathcal{X},x) \to \widehat{\pi_1 (\mathcal{X})}$. Then the covering stack of $\mathcal{X}$ corresponding to 
$H$ is a complex orbifold $\mathcal{X}'$  with a finite \'etale covering $\mathcal{X}' \to \mathcal{X}$. Consider the moduli space $X'$ of $\mathcal{X}'$.
It is a normal variety with quotient singularities at worst which carries an action of $G=\pi_1(\mathcal{X})/H$. This gives a map of stacks with identical moduli spaces $\mathcal{X} \to [X'/G]$ with the property that any finite dimensional (or finite image)
representation of $\pi_1(\mathcal{X})$ factors through $\pi_1([X'/G])$. 

\begin{defi} We define $S(N)=S_{\pi}(N)$ for every $N\in \N$ by applying this construction to $\mathcal{S}_{\pi}[N]$,  $S(N)=S'_{\pi}[N]$ being a normal complex projective surface with isolated quotient singularities. 
 
\end{defi}
 
\begin{rem}
 The minimal resolution of singularities  $\widehat{S(N)} \to S(N)$ induces an isomorphism on $\pi_1$. 
\end{rem}

\begin{rem} 
 It is proposed  in \cite{BK} to investigate 
the surfaces $S(N)$ to find counterexamples to Shafarevich conjecture on holomorphic convexity among them.  
\end{rem}


\section{Using TQFT representations for the Shafarevich Conjecture}

\subsection{The Birman exact sequence}
All the above fiber sequences come from the fiber sequence of $f_1$, which by means 
of the isomorphisms $\pi_1(\mathcal{M}_g^n)\cong \PMod(\Sigma_g^n)$ can be identified with 
the Birman exact sequence: 
\[
\xymatrix{
1\ar[r] &   \pi_1(\Sigma_g) \ar[r]^{i_*} \ar@{=}[d] & \pi_1(\Mgone) \ar[r]^{f_{1*}}\ar[d]^{\cong} &  \pi_1(\Mg) \ar[r]\ar[d]^{\cong}& 1 \\
1\ar[r] &   \pi_1(\Sigma_g) \ar[r]^{pp}                   & \Mod (\Sigma_g^1) \ar[r]    & \Mod(\Sigma_g) \ar[r] & 1, }
\]
where $pp$ is Birman's point pushing homomorphism.

The point pushing of a simple closed loop is  a product of two Dehn twists along disjoint simple curves in $\Mod(\Sigma_g^1)$. 
Let $S(\Sigma_g)$ be the set of conjugacy classes of of $\pi_1(\Sigma_g)$ which can be represented 
by simple closed curves on $\Sigma_g$. 
Also Dehn twists are preserved by 
point forgetting maps.

\begin{lemma}
The map $f_1$ lifts to a map of stacks $f_1[p]:\Mgobaran[p] \to \Mgbaran[p]$
compactifying $f_1$ inducing at the level of fundamental groups the canonical map 
$\Mod(\Sigma_g^1)/\Mod(\Sigma_g^1)[p]\to \Mod(\Sigma_g)/\Mod(\Sigma_g)[p]$. 
\end{lemma}
\begin{proof}
This follows from Lemma \ref{pullback}, the fact that the root construction behaves well under pullback (see \cite{cadman2007}) and the universal property of canonical stacks \cite[Theorem 4.6]{FMN}. 
\end{proof}

It follows from Theorem \ref{Kahlerorbifold} that the Birman exact sequence gives rise to a level $p$ analogue, as proved in \cite{AF}:
\[
\xymatrix{1\ar[r] &   BT(\pi_1(\Sigma_g), p) \ar[r]  & \pi_1(\Mgobaran[p]) \ar[r]^{f_{1}[p]_*}&  \pi_1(\Mgbaran[p]) \ar[r] & 1 }
\]
where 
\[ BT(\pi_1(\Sigma_g),p)=\pi_1(\Sigma_g)/\langle\langle x^p, x\in S(\Sigma_g) \rangle\rangle,  
\]

This is of course the fiber sequence for  the map of stacks $f_1[p]:\Mgobaran[p] \to \Mgbaran[p]$. 
The group $BT(\pi_1(\Sigma_g), p)$ is actually 
the fiber group $I_p(\mathbf{C}_{f_1})$. 
Here we abuse notation, such a fiber group  arises from a family $\pi$ such that $\rho$ is surjective.

By construction, the map $C\to \Mgbar$ lifts to a map of stacks 
$\mathcal{C}[p]\to \Mgbar[p]$. 
Comparing with Proposition \ref{boka}, we find a diagram: 
\[\xymatrix
{1\ar[r] &  I_p(\mathbf{C}_{\pi}) \ar[r] \ar@{>>}[d] & \pi_1(\mathcal{S}_{\pi}[p])\ar[r]\ar[d] &  \pi_1(\mathcal{C}[p]) \ar[r]\ar[d]& 1 \\
1\ar[r] &   B(\pi_1(\Sigma_g),p) \ar[r] & \pi_1(\Mgobaran)[p] \ar[r]&  \pi_1(\Mgbaran)[p] \ar[r]& 1  }\]  
with the leftmost vertical map being given by the natural quotient map attached to $\mathbf{C}_{\pi}\subset \mathbf{C}_{f_1}=S(\Sigma_g)$.

\subsection{Steinness of the universal covering space of most Bogomolov-Katzarkov surfaces}

The starting point of this project was the realization that  \cite{KoSa} implies the following result for large odd numbers $N$: 

\begin{theo}\label{kosa}
 If $g\geq 2$, $N\ge 5$ is odd, $N$ even and $N/2 \ge 5$ is odd, or $N\equiv 0 ({\rm mod}\; 4)$ and $N\ge 20$ and $(g,N)\not=(2,24)$   
 then there exists a complex finite dimensional linear representations $\rho$ of $\pi_1(\mathcal{S}_{\pi}[N])$ 
such that $\rho \circ \iota_*$ has infinite image containing a 
free group on two generators, where $\iota$ is the inclusion of the general fiber into 
$\mathcal{S}_{\pi}[N]$. 
 \end{theo} 
 \begin{proof} If $N$ is odd, define $p=N$. By construction there is map of stacks $i:\mathcal{S}_{\pi}[p] \to \Mgobaran[p]$ and we can 
 define $\rho=\rho_{p, (2)} \circ \iota_*$. The result is then an
 immediate consequence of Propositions \ref{factori} and \ref{inffib}. 
 
 If $N$ is even, define $p=\frac{N}{2}$. By construction there is map of stacks $i:\mathcal{S}_{\pi}[N] \to \Mgobaran[2p, p/g.c.d(p,4)]$ and we can 
 define $\rho=\rho_{p, (2)} \circ \iota_*$. The result is then an
 immediate consequence of Propositions \ref{factori} and \ref{inffib}. 
 \end{proof}

 \begin{coro}\label{holo}
Under the hypotheses of Theorem \ref{kosa},     $S(N)$
has a holomorphically convex infinite Galois covering space
  with the property that the codomain of its Cartan-Remmert reduction is $2$-dimensional. 
 \end{coro}
 
 \begin{proof} Recall that 
 the Cartan-Remmert reduction of a holomorphically convex normal complex space  $T$ is the unique proper holomorphic map $T\to U$
 such that $U$ is a normal Stein space. 
 
 The main theorem in \cite{EKPR} states that, given a complex projective manifold $X$ and  $M\in \N$, there is a Cartan-Remmert reduction:
 \[ \widetilde{s_M}: H_M\backslash   \widetilde{X^{univ}} \to  \widetilde{U(M)} \quad
 \mathrm{with} \ H_M =\bigcap_{\rho: \pi_1(X) \to GL_M(\C)} \ker(\rho) .\]
Since $\Gamma_M=\pi_1(X)/H_M$ acts in a properly discontinuously, this map descends to a proper holomorphic map, the $GL_M$-th Shafarevich morphism of $X$: 
 
\[Sh_M: X \to Sh_M(X)= \Gamma_M \backslash \widetilde{U(M)}.\]
 The fibers of $\widetilde{s_M}$ are the connected components of the lifts of the fibers of $Sh_M$ which are characterized as the maximal connected subvarieties  $Z \subset X$
 such that $\mathrm{Im}(\pi_1(Z)) \to \Gamma_M)$ is finite. In particular, when all positive dimension irreducible  subvarieties $Z$ satisfy: 
 \[\mathrm{Im}(\pi_1(Z^{norm})) \to \Gamma_M) \ 
 \mathrm{is\  infinite},\] where $Z^{norm}\to Z$ is the normalization,    then $\widetilde{s_M}=\mathrm{Id}$ and $\widetilde{U(M)}$ is Stein, 
 and when all positive dimensional irreducible  subvarieties $Z$ through the general point satisfy this property, then $\dim{\widetilde{U(M)}}=\dim(X)$.

 In the situation of the corollary,  $\pi_1(S(N))$  has a finite degree linear representation with infinite image that comes from the base of the fibration $\pi$ whose image is  a non parabolic curve. Assume that its degree is $M_0$.   
Therefore, whenever $Z$ is a connected  closed subspace
such that all linear representations of $\pi_1(S(N))$ have  finite images in $\Gamma_M$ for $M\ge M_0$ upon restriction to $\pi_1(Z)$, then $Z$ is contained in the fibers of $\pi$.  Then Theorem \ref{kosa}  gives a finite degree  representation, say of degree $M_1$,  which has infinite image on  the general fiber of $\pi$. Hence, if  $\pi_1(Z) \to \Gamma_M$ has finite image for $M\ge \max(M_0,M_1)$, 
then  $Z$ lies in a special fiber of $\pi$. This implies the corollary. 
  
 \end{proof}

\begin{theo}\label{shaf} 
Assume that  $g\geq 2$ and $N\not\in\{1, 2, 3, 4, 6, 8, 12, 16, 20\}$ and   
$(g,N)\neq (2,24)$. Then 
the universal covering space of $S(N)$ is Stein. 
\end{theo}
\begin{proof}
We will assume $N=p$ is odd in order to simplify the notation.
We claim that 
there is a complex representation $\rho$ of $\pi_1(\mathcal{S}_{\pi}[p])$ such that 
the image of $\pi_1(Z)$ is infinite, for  every irreducible component $Z$ of  the special fiber of the natural morphism to a curve  $ \psi:S(p)\to C'_p$ 
obtained as the Stein factorization of the  composition of $S(p) \to S \buildrel{\pi}\over\longrightarrow C$. We will show that the representation $\rho_{p, (2)}$ is convenient. 

Since the fibres of $\pi$ are stable, 
the complement of  the singular set of  each such component $Z$ contains as a Zariski dense open subset a finite \'etale covering of  the complement of the punctures in a stable curve with $n>0$ marked points.
In particular it contains an essential subsurface $\Sigma$ with boundary which is  homeomorphic 
to a pair of pants $\Sigma_{0,3}$ or to a $\Sigma_{1,1}$.  This subsurface $\Sigma$  can be deformed to a
subsurface of the general fiber of $\pi$ since $\pi$ 
is smooth near $\Sigma$. In particular it can be identified with a subsurface of a fiber of $f_1$. 
Then we can compute  $\rho(\pi_1(Z))$ by using the restriction of $\rho_{p,(i_1)}$ 
to the fundamental group $\pi_1(\Sigma_g)$ of the general fiber.

Since $Z$ contains a finite \'etale cover of the $\Sigma$, the representation $\rho$ has infinite image on the fundamental group of this  subsurface and hence on the fundamental group of $Z$, by Propositions \ref{factori} and \ref{inffib}. 

Hence the Shafarevich morphism of $S(p)$ used in the proof of Corollary \ref{holo} cannot contract any component of the preimage of $Z$ in $S(p)$. Since the universal covering space of a 
Stein manifold is Stein, this concludes the proof.  

To do the proof in the even case, one has to use the fact that if $X, Y$ are normal spaces, $Y$ being Stein, $X\to Y$ is a holomorphic map has the property that the preimage of 
a small open subset $\Omega\subset Y$ is a disjoint union of proper finite ramified
covering spaces of $\Omega$, $X$ is Stein (\cite{LeB}).

In fact \cite{EKPR} is not needed here, the earlier result \cite{KR}  allows to conclude from the semisimplicity of $\rho_{p, (i_1)}$. 

\end{proof}

\subsection{ The stack $\mathcal{S}_{\pi}[2p,\frac{p}{{\rm g.c.d.}(4,p)}]$} 
In the proofs above we can use the quantum representations $\rho_p$, for even $p\geq 10$, to show that 
$\mathcal{S}_{\pi}[2p,\frac{p}{{\rm g.c.d.}(4,p)}]$  has an holomorphically convex universal covering 
which is Stein.

 \subsection{Unstable families of curves and end of proof of 
 Theorem \ref{bogo-katza}}

If $\pi:S\to C$ is not supposed to be stable but has only nodal singularities on its fibers as in \cite{BK} then any maximal chain $\Gamma$ of 
unstable components is either a $(-1)$-rational curve or a chain of  $n_{\Gamma}$ $(-2)$-rational curves for some $n_{\Gamma}\in \N^*$. It can be blown down and one gets a fibered 
surface $\pi: S^*\to C$ with stable singular fibers. Each singular point $s\in S^{*, sing}$  is an $A_{n(s)}$-singularity for some $n(s) \in \mathbb{N}^*$
and a node of singular fiber it belongs to. 

One can do the same construction as in \cite{BK} for $S$. Let $D$ be the sum of the singular fibers, each component being counted with multiplicity one, so that $\pi^* B=D$ where $B\subset C$ is the set of singular values of $\pi$. 

However, if there is a $(-1)$-rational curve, 
$\mathcal{S}_{\pi} [N]$ need not be uniformizable. This situation occurs when we blow up some smooth points of the singular fiber of a smooth stable fibration $\pi: S^* \to C$. 
The orbisurface $S(N)$ is then the canonical stack of a weighted blow up of $S^*(N)$ of weights $(1,N)$ hence has a $A_{N-1}$ singularity.
Thus, without loss of generality we can assume  $S$ to be minimal which rules out these $(-1)$-rational curves. 

Then,  we have a 
commutative diagram: 
\[
\xymatrix{ S\ar[r] \ar[rd] & S^*\ar[d]_{\pi}\ar[r] & \Mgobar \ar[d]^{f_1}\ar@{=>}[ld]\\
 &C\ar[r]^{\psi} & \Mgbar}
\]
 
The square is cartesian, i.e.: $S^*\simeq C\times_{\psi,f_1} \Mgobar$,   $\pi$ being equivalent to $\psi^* f_1$. Denote by $D^*$ the sum of the singular fibers of $\pi: S^* \to C$. 
The stack $S^*(\sqrt[N]{D^*})$ has a quotient singularity of type $pn(s)$ over $s\in  S^{*, sing}$
and we may set $\mathcal{S}^*_{\pi}[N]:=S(\sqrt[N]{D^*})^{can}$. It is not clear whether $\mathcal{S}^*_{\pi}[N]$ is uniformizable. The root stack being a natural construction,  we have a map of stacks 
$(S^*\setminus S^{*,sing} )\sqrt[N]{D}\to\overline{\mathcal{M}_g^{1 \ 0}}[N]$ with the notation as in Remark \ref{codim2calculation}. It induces a morphism on the fundamental groups and 
since $\pi_1((S^*\setminus S^{*,sing})(\sqrt[N]{D^*})) \simeq \pi_1(\mathcal{S}^*_{\pi}[N])$ and $\pi_1(\overline{\mathcal{M}_g^{1 \ 0}}[N]) \simeq \pi_1(\overline{\mathcal{M}_g^{1}}[N])$ a group morphism
$\pi_1 (\mathcal{S}^*_{\pi}[N]) \to \pi_1({\overline{\Mgonean}} [N]) $ factorizing the natural morphism
$\pi_1(S^*\setminus D^*) \to \pi_1({\overline{\Mgonean}} [N]) $.  Consider now the normal surface $S^{*}_{\pi}[N]'$ constructed at the end of section 5. 
The quantum representations give rise to representations of $\pi_1 (S^{*}_{\pi}[N]')$ (and of $\pi_1(\widehat{S^{*}_{\pi}[N]'}) $ ) and the proof of Theorem \ref{shaf} applies to the effect that
$S^{*}_{\pi}[N]'$ has a Stein universal covering space and $\widehat{S^{*}_{\pi}[N]'}$ has a holomorphically convex universal covering space which resolves the singularities of 
$\widetilde{S^{*}_{\pi}[N]'}^{univ}$. 

The same construction gives actually a group morphism $\pi_1( \mathcal{S}_{\pi}[N] ) \to \pi_1({\overline{\Mgonean}} [N])$ factorizing the natural morphism
$\pi_1(S\setminus D) \to \pi_1({\overline{\Mgonean}} [N]) $. It descends to a map of fundamental groups $\pi_1(\widehat{S(N)})= \pi_1(S(N)) \to \pi_1({\overline{\Mgonean}} [N]) $. Then, we can 
construct the $GL_N$ Shafarevich morphism, which contracts exactly the connected components of the preimage of the chain of $(-2)$ rational curves by the natural map to $S$, which have finite fundamental groups. 
This is enough to imply the holomorphic convexity of the univeral covering spaces of $\widehat{S(N)}$ and $ S(N)$. 

\subsection{Concluding remarks}

If $\pi:S\to C$ is an irrational pencil of curves (resp. a rational pencil) on a smooth projective surface and has more complicated singularities, such as non reduced components of the singular fibers, 
we can use the semistable reduction theorem \cite{DMu}, 
to find $C'\to C$ a ramified covering, which ramifies at  the unstable singular fibers 
(resp. one may have to introduce a further ramification point on the smooth locus)  to get a normal projective surface $S'$ birational to $S\times_C C'$ 
with an irrational  pencil (resp. a general pencil of curves) with only stable curves as scheme theoretic singular fibers and one is reduced to the previous situation.

So in order to find a   counterexample to the Shafarevich conjecture along the lines of \cite{BK}, the only remaining possibilities with a base change having a common
ramification index, require a ramification index  $N \in \{1, 2,3,4,6,8,12, 16,20,24\}$ ($N=1$ corresponds to no base change at all) or allowing a fibration with non-reduced fibers.

Another
possibility is to perform a
fiber product with $\Mgobaran[\mathbf{k}]$  with general weight vector $\mathbf{k}$. The fundamental groups of 
the $\Mgobaran[\mathbf{k}]$  with general weight vector $\mathbf{k}$ do not seem to be easy to study.
Let $\Pi$ be the fundamental group of a hyperbolic surface (possibly punctured). 
Let $\{\gamma_i\}$ be a set of isotopy classes of simple closed curves whose orbits 
under the action of the mapping class group are pairwise disjoint.   
Then, Wise and Bajpai proved in \cite{Baj,Wise} that there is some $d\geq 1$ such that for every positive integral vector 
$\mathbf{k}=(k_i)$ there exists some finite quotient $\Pi\to \Pi/K$ in which images of  
the classes $\gamma_i$ have orders $dk_i$, respectively. 
We don't know whether the kernel $K$ might be chosen to be invariant with respect to the 
mapping class group action. If it could and additionally $K$ satisfies  
Boggi's condition from \cite[Thm. 3.9]{Boggi} (for example, if $K=[L,L]L^{\ell}$, for some $\ell\geq 3$, where $L$ is invariant subgroup of $\Pi$) then 
the compactified Deligne-Mumford stack of curves endowed with 
Looijenga's levels associated to $K$ and $m\geq 2$ would be smooth 
and representable. This further implies that  $\Mgbaran[md \mathbf{k}]$ is uniformisable.

In a similar vein, one may as in \cite{BK} start with a stable fibration $\pi:S \to C$ with reduced fibers and try and study the fundamental groups of the canonical stacks 
obtained by assigning different multiplicities to the fibers. We don't known how to decide whether the fiber group is finite or infinite 
when the multiplicities are coprime.

\end{document}